\documentclass[12pt,a4paper]{amsart}

\usepackage{caption}
\usepackage{a4}
\usepackage{amssymb}
\usepackage{amsmath,amsthm,geometry,enumerate}
\usepackage{psfrag}
\usepackage{amscd}
\usepackage{xspace}
\usepackage[all]{xy}
\usepackage{longtable}
\usepackage{tikz}
\usepackage{hyperref}
\usetikzlibrary{shapes.geometric}

\def\cf{\textit{cf.}\kern.3em}

\def\resp{\textit{resp.}\kern.3em}
\renewcommand{\k}{\kern2pt}
\numberwithin{equation}{section} 

\let\emptyset\varnothing
\DeclareMathOperator{\pic}{Pic}

\DeclareMathOperator{\spec}{Spec}

\DeclareMathOperator{\Vrt}{Vert}

\newtheorem{proposition}[equation]{Proposition}
\newtheorem{theorem}[equation]{Theorem}
\newtheorem*{theorem*}{Theorem}
\newtheorem{corollary}[equation]{Corollary}
\newtheorem{lemma}[equation]{Lemma}
\theoremstyle{definition}
\newtheorem{definition}[equation]{Definition}

\newtheorem{remark}[equation]{\textbf{Remark}}
\newtheorem{notation}[equation]{\textbf{Notation}}

\newtheorem{example}[equation]{\textbf{Example}}

\newcommand{\m}[2]{$\mathcal{M}_{#1,#2}$}

\newcommand{\Pp}[1]{\mathbb P^{#1}}
\newcommand{\C}{\mathbb C}
\newcommand{\R}{\mathbb R}
\newcommand{\Q}{\mathbb Q}
\newcommand{\Z}{\mathbb Z}
\newcommand{\co}{\colon\thinspace}

\newcommand{\Mm}[2]{\mathcal{M}_{#1,#2}}

\newcommand{\Mmb}[2]{\overline{\mathcal{M}}_{#1,#2}}
\newcommand{\NR}[1]{{\mathcal{M}}^{\operatorname{NR}}_{1,#1}}
\DeclareMathOperator{\NRop}{NR}

\newcommand{\Jm}[3]{{\mathcal{J}}^{#1}_{#2,#3}}

\newcommand{\Jmb}[3]{\overline{\mathcal{J}}^{#1}_{#2,#3}}

\newcommand{\Cmb}[2]{\overline{\mathcal{C}}_{#1,#2}}
\newcommand{\CNR}[1]{{\mathcal{C}}^{\operatorname{NR}}_{1,#1}}
\newcommand{\HS}{\mathsf{HS}}
\newcommand{\bfa}{\mathbf{a}}
\newcommand{\bfb}{\mathbf{b}}
\newcommand{\bfc}{\mathbf{c}}
\newcommand{\Simp}[2]{\operatorname{Simp}^{#1}({#2})}

\newcommand{\coh}[2]{H^{#1}({#2})}
\newcommand{\cohc}[2]{H_c^{#1}({#2})}

\newcommand{\pu}{\bullet}
\newcommand{\hodgechar}[2]{\mathbf{e}_{\HS_\Q}^{\s_{#1}}(#2)}
\newcommand{\s}{\mathfrak S}

\usepackage{color}
\definecolor{forestgreen}{rgb}{0.13, 0.55, 0.13}

\newcounter{note}
\begin{document}

\def\fcj{fine compactified Jacobian\xspace}
\def\fcuj{fine compactified universal Jacobian\xspace}

\def\fcjs{fine compactified Jacobians\xspace}

\def\fcujs{fine compactified universal Jacobians\xspace}

\date{}
\title{Geometry of genus one fine compactified universal Jacobians}
\author[N. Pagani]{Nicola Pagani}
\address{Department of Mathematical Sciences, Mathematical Sciences Building, Liverpool L69 7ZL, United Kingdom}
\email[Nicola Pagani]{pagani@liv.ac.uk}

\author[O. Tommasi]{Orsola Tommasi}
\address{Dipartimento di Matematica ``Tullio Levi-Civita'', University of Padova, via Trieste 63, I-35121 Padova, Italy}
\email[Orsola Tommasi]{tommasi@math.unipd.it}

\begin{abstract}
 We introduce a general abstract notion of fine compactified Jacobian for nodal curves of arbitrary genus.
 
 We  focus on genus $1$ and prove combinatorial classification results for fine compactified Jacobians in the case of a single nodal curve and in the case of the universal family $\Cmb 1n / \Mmb 1n$ over the moduli space of stable pointed curves. 
 We show that if the fine compactified Jacobian of a nodal curve of genus $1$ can be extended to a smoothing of the curve, then it can be described as the moduli space of stable sheaves with respect to some polarisation.
 In the universal case we construct new examples of genus $1$ fine compactified universal Jacobians.

Then we give a formula for the Hodge and Betti numbers
 of each genus $1$ fine compactified universal Jacobian $\Jmb d 1n$ and prove that their even cohomology is algebraic. 
\end{abstract}
\maketitle
\setcounter{tocdepth}{2}

\section{Introduction}

A classical construction from XIX century mathematics associates with 
every nonsingular complex projective curve its \emph{Jacobian}, a complex projective variety of dimension equal to the genus of the curve. A similar construction is available for singular curves, but the resulting Jacobian variety fails in general to be compact. The natural problem of compactifying the Jacobian variety of singular curves has seen multiple approaches and generated a large body of work starting from the mid XX century.  

Inspired by this literature, in the present paper we introduce a general notion of fine compactified Jacobian for families of stable complex curves, and then investigate the geometry and combinatorics of the objects constructed using this definition. We focus on the genus $1$ case where we produce complete classification results, and an explicit description of the topology of these objects.

Let us recall that for any fixed integer $d$, a nonsingular curve $C$ admits a degree~$d$ Jacobian, which is the moduli space $J_C^d$ of degree $d$ line bundles on $C$. 
The construction of the Jacobian works for smooth families of curves. For all $(g,n)$ with $2g+n>2$ and all integers $d$, there exists a universal Jacobian $\Jm{d}{g}{n}$ endowed with a smooth fibration $\Jm{d}{g}{n} \to \Mm{g}{n}$ whose fibre over a pointed curve $(C, p_1, \ldots, p_n)$ is $J_C^d$. The problem of extending the universal Jacobian to the universal family over the Deligne--Mumford compactification $\Mmb gn$ has a long history, beginning with Caporaso \cite{caporaso}, Simpson \cite{simpson} and Pandharipande \cite{panda} in the $n=0$ case. 
Since the moduli space of line bundles of degree $d$ on a nodal curve $C$ is not compact as soon as $C$ has nonseparating nodes,
 the problem of compactifying the universal Jacobian is intimately related to the problem of compactifying the Jacobians of nodal curves. 

Here we consider compactifications whose limit points parametrise degenerations of line bundles, which we will take to be rank~$1$ torsion-free simple sheaves (for brevity we will just refer to them as \emph{simple sheaves} in the rest of this introduction). 

The moduli space of simple sheaves of some fixed degree $d$ was constructed by Altman--Kleiman in \cite{altman80} and then by Esteves \cite{esteves} as an algebraic space that satisfies the existence part of the valuative criterion of properness but that fails to be separated and of finite type.  Esteves' construction applies to families of nodal curves over a scheme, and by  \cite{melouniversal} also to the universal family $\Cmb{g}{n}/\Mmb{g}{n}$ over the moduli stack of stable $n$-pointed curves of genus~$g$.

In Definition~\ref{def:fcj} we introduce the notion of a degree~$d$ \fcuj $\Jmb dgn$ as an open and proper substack of the moduli stack of simple sheaves for the universal family over $\Mmb gn$. The fibre  $\overline{\mathcal{J}}^d(X)$ of $\Jmb dgn \to \Mmb gn$ over the point representing the stable curve  $X$ is an example of what we will call a ``\fcj{}'': a connected, open and proper subscheme of the moduli space of simple sheaves on $X$. The compactifications of the universal Jacobian constructed in \cite{melouniversal} are examples of \fcujs{}; the same examples were obtained in \cite{kp3} by generalising to the universal curve over $\Mmb gn$ the approach used by  Oda--Seshadri \cite{oda79} for the case of a single nodal curve. 

After these general definitions, we focus our attention on the first nontrivial case, i.e. the case of genus $g=1$.  It is a classical result that the Jacobian variety of an irreducible nodal curve of arithmetic genus $1$ is  isomorphic to the curve itself, and that the Jacobian variety of a nonsingular curve of genus $0$ is a point. Something similar holds more in general for \fcjs. 

Let us recall that every nodal curve $X$ of genus~$1$ is isomorphic to a unique necklace subcurve $X'$ of $X$ with rational tails attached to it, where by a necklace curve we mean a curve of arithmetic genus $1$ which is either irreducible, or which consists of rational components glued each to the next in a cyclic way. 
 All simple sheaves in a fixed \fcj have the same degree when restricted to the components of a rational tail of $X$, hence rational tails contribute trivially to the geometry of \fcjs (Corollary~\ref{forgettails}). In fact, as it turns out, all \emph{smoothable} \fcjs (see Definition~\ref{smoothable})  are isomorphic to the necklace subcurve~$X'$ (Proposition~\ref{prop:I_n}). 
 
More specifically, there is a natural bijection between the singular locus of $\overline{\mathcal J}^d(X)$ and that of $X'$: the singular locus of $\overline{\mathcal J}^d(X)$ consists of sheaves that fail to be locally free exactly at one node of $X'$, and the bijection maps each such sheaf to the corresponding node. This bijection induces a cyclic ordering on the nodes of $X'$, and our first result is that this cyclic ordering essentially identifies the fine compactified Jacobian of $X$:

\begin{theorem*} (Corollary~\ref{cor:fcjs})
There is a natural bijection between the set of smoothable \fcjs{} of a nodal genus $1$ curve $X$ up to equivalence by translation, and the set of cyclic orderings of the  nonseparating nodes of $X$.
\end{theorem*} 
\noindent Here with \emph{equivalent by translation} for two compactified Jacobians $\overline {\mathcal{J}}(X)$ and $\overline{\mathcal{J}}'(X)$ of a curve $X$  we mean that there exists a line bundle $L$ on $X$ such that tensoring by $L$ defines an isomorphism $\overline{ \mathcal{J}}(X) \xrightarrow{\sim} \overline{\mathcal{J}}'(X)$.

A consequence  is that each smoothable fine compactified Jacobian of a genus~$1$ nodal curve arises %
as the moduli space of simple sheaves that are stable with respect to some polarisation, and the geometry and combinatorics of those have been studied in \cite[Section~7]{meravi}. 
We also exhibit examples of \fcj, i.e. of  connected, open and proper subschemes of the moduli space of simple sheaves on a genus $1$ stable curve, that are \emph{not} smoothable.

Our second  result is a combinatorial classification of all genus $1$ fine compactified \emph{universal} Jacobians.

\begin{theorem*}(Theorem~\ref{f-and-g})
The degree~$d$ \fcujs  over $\Mmb 1n$ are naturally identified with 
the pairs $(f,g)$ where:
\begin{enumerate}
    \item the function $g$ is an integer valued function on the set of all subsets of  $\{1,\dots,n\}$ containing at least $2$ elements, and
\item the function $f$ is an integer valued function on the set of all nonempty subsets of $\{1, \ldots,~n~-~1\}$
that satisfies the following mild superadditivity condition:
\[
0\leq f(I\cup J) - f(I) - f(J) \leq 1 \text{ for all }I,J \text{ such that } I \cap J = \emptyset.
\]
\end{enumerate} 

\end{theorem*}
The functions $f$ and $g$ encode the information of the bidegrees of all elements $(X,L)$ of the given \fcuj{} $\Jmb d1n$ such that $L$ is a line bundle on a stable curve $X$ with two nonsingular irreducible components. The main point of the theorem is to show that from these data one can uniquely reconstruct the \fcuj.

As a consequence of this combinatorial description, we show in Example~\ref{exotic} that 
 for all $n\geq 6$ there are new examples of genus $1$ fine compactified universal Jacobians that cannot be obtained using the description of \cite{kp3} (or equivalently \cite{melouniversal}). 
 We also observe in Remark~\ref{groupaction} that for fixed $n$ there is only a \emph{finite} number of degree~$d$ fine compactified universal Jacobians  modulo translation by line bundles of relative degree $0$ on the universal curve $\Cmb 1n / \Mmb 1n$. 

Finally, we describe the rational cohomology groups of all genus~$1$ fine compactified universal Jacobians, together with their Hodge structures. Since the symmetric group $\s_n$ acts on every $\Jmb d1n$ by permuting the marked points, the cohomology groups carry a structure as an $\s_n$-representation, which we determine as well. 
It turns out that all genus $1$ \fcujs for fixed $n$  have the same rational cohomology groups, and the latter can be explicitly described in terms of the rational cohomology of the  moduli spaces of stable curves of genus $0$ and $1$. 
An analogous pattern of independence of the  choice of \fcj was observed in  \cite{mishvi} for \fcjs defined by a polarisation over a single nodal curve of arbitrary genus.
\begin{theorem*} (Theorem~\ref{main})
 The rational cohomology of any genus~$1$ fine compactified universal Jacobian $\Jmb{d}{1}{n}$ is described explicitly  as a graded vector space with Hodge structures and $\s_n$-representations by Formula~\eqref{mainformula} in terms of the known generating functions for the Hodge Euler characteristics of $\Mm1k$, $\Mm0m$ and $\Mmb0\ell$. 
\end{theorem*}

(We stress that our main result applies to \emph{all} genus~$1$ fine compactified universal Jacobians, not just those that are obtained as moduli of simple sheaves that are stable with respect to some polarisation.)

Theorem~\ref{main} is obtained by introducing a refinement of the stratification of $\Mmb 1n$ by topological type via the forgetful map $\Jmb d1n \to \Mmb 1n$. It is  a consequence of Petersen's results in \cite{petersen} that the boundary strata classes span  the even cohomology of $\Mmb{1}{n}$ (in particular, the even cohomology of $\Mmb{1}{n}$ is all algebraic). In Corollary~\ref{boundary} we observe that an analogous result holds for the even cohomology of $\Jmb{d}{1}{n}$. We describe the stratification of the latter moduli space in Corollary~\ref{strata}. %

Finally, let us remark that,  whilst the additive rational cohomology of $\Jmb{d}{1}{n}$ does not depend on the particular \fcuj, we expect that there exist nonisomorphic \fcujs for fixed $n$. For example, it is shown in \cite[Section~6.2]{kp3} that for all $n\geq 4$ there exist   genus $1$ \fcujs that are not isomorphic over $\Mmb 1n$ (i.e. there exists no isomorphism that respects the forgetful morphisms).

\subsection{Acknowledgements}
The first named author owes a deep debt of gratitude  to Jesse Kass for the many discussions and clarifications on compactified Jacobians, and for contributing in a substantial way to shaping the definition of a fine compactified Jacobian of a nodal curve. He is also very grateful to Margarida Melo and Filippo Viviani for helpful discussions.

The second author would like to thank Jonas Bergstr\"om for sharing with her his tables on the Hodge Euler characteristic of $\Mm1n$. 

We are grateful to Jonathan Barmak for his help with the combinatorics of Example~\ref{exotic}. We would also like to thank Dan Petersen for a helpful discussion of his work \cite{petersen} and Alex Abreu, Eduardo Esteves and Andrea Ricolfi for providing useful comments on an earlier version of this paper.

Finally, we would like to thank an anonymous referee for their suggestions to improve the paper. 

During the early stages of this project NP was supported by EPSRC grant EP/P004881. OT is a member of INdAM-GNSAGA and was partially supported during the preparation of this paper by the BIRD-SID 2019 grant \emph{Structures on the cohomology of moduli spaces and stratified spaces} of the University of Padova.

\subsection{Notation}

We work over the category of schemes of finite type over $\mathbb{C}$. 

A \emph{nodal curve} is a reduced and connected projective scheme of dimension $1$ over~$\C$ with singularities that are at worst ordinary double points. 
If $X$ is a nodal curve, a \emph{subcurve} $Y$ of $X$ is a connected union of irreducible components of $X$. 
A subcurve $Y \subseteq X$ is called a \emph{rational tail} if the arithmetic genus of $Y$ equals zero and $Y \cap  \overline{Y^c}$ consists of $1$ point (here $Y^c=X\setminus Y$ denotes the complement of $Y$ in $X$, and $\overline{Y^c}$ is the closure of the latter in $X$). 

A coherent sheaf on a nodal curve $X$ has \emph{rank~$1$} if its localisation at each  generic point of $X$ has length $1$. It is \emph{torsion-free} if it has no embedded components. 
If the stalk of a torsion-free sheaf $F$ over $X$ fails to be locally free at a point $P\in X$, we will say that $F$ is \emph{singular} at $P$.
If $F$ is a rank~$1$ torsion-free sheaf on $X$ we say that $F$ is \emph{simple} if its automorphism group is $\mathbb{G}_m$ or, equivalently, if removing from $X$ the singular points of $F$ 
 does not disconnect $X$. 
A \emph{family of nodal curves} over a $\mathbb{C}$-scheme $S$ is a proper and flat morphism $X \to S$ whose fibres are nodal curves. We will always make the additional hypothesis that a family $X/S$ admits a section in the $S$-smooth locus of $X$.

If $T$ is a $S$-scheme, a \emph{family of rank~$1$ torsion-free simple sheaves} parametrised by $T$ over a family of curves $X \to S$ is a  coherent sheaf $F$ of rank~$1$  on $X\times_S T$, flat over $T$, whose fibres over the geometric points are torsion-free and simple.

If $(X,p_1,\dots,p_n)$ is a stable $n$-pointed  curve, we will denote by $\Gamma(X)$ its dual graph, i.e. the labelled graph with vertices, edges and half-edges corresponding to the irreducible components of $X$ labelled by their geometric genus, nodes of $X$ and marked points labelled from $1$ to $n$, respectively.

The dual graph $\Gamma(X)$ is an object of the category of stable graphs of genus~$g$ with $n$ marked half-edges (see \cite[Section~2]{kp3} for their definition and more details).  A morphism $f\colon\Gamma \rightarrow \Gamma'$ of stable graphs is defined as a sequence of strict contractions of edges (in the sense of \cite[Section~2.2]{kp3}) followed by an isomorphism of graphs. For every $g$ and $n$ we will choose a representative for each of the finitely many isomorphism classes of stable graphs of genus~$g$ with $n$ half-edges and denote by  $G_{g,n}$ the full subcategory on the chosen objects. 

If $F$ is a rank~$1$ torsion-free sheaf on a nodal curve $X$ with irreducible components $X_i$, we denote by $F_{X_{i}}$  the maximal torsion-free quotient of $F \otimes \mathcal{O}_{X_{i}}$, and then define the \emph{multidegree} of $F$ by \[{\deg}(F) := (\deg(F_{X_{i}})) \in \mathbb{Z}^{\operatorname{Vert}(\Gamma(X))}.\]  We define the \emph{(total) degree} of $F$ to be $\deg_X(F):=\chi(F)-1+p_a(X)$ where $p_a(X)= h^1(X, \mathcal{O}_X)$ is the arithmetic genus of $X$. The total degree and the multidegree of $F$ are related by the formula $\deg_X(F) = \sum \deg_{X_i} F - \delta(F)$, where $\delta(F)$ denotes the number of singular points of $F$.

\label{notation}
\section{Fine compactified Jacobians} 
\label{generalfine}

The main topic of this paper is the study of \fcujs{} in genus~$1$. We start this section by introducing the notion of a \fcuj{} as a substack of the moduli stack of rank~$1$ torsion-free simple sheaves for the universal family on the moduli stack of curves with marked points. 
A more thorough analysis of the theory of \fcujs in genus larger than $1$ will be developed in \cite{kasspa2}.

Let $X/S$ be a family of nodal curves over a scheme $S$, which we always assume to admit a section in the $S$-smooth locus of $X$. %
Then its \emph{generalised Jacobian}  $\mathcal{J}^{\mathbf 0}(X/S)$ is the moduli space of line bundles of fibrewise degree $0$ on all irreducible components. It is well known that $\mathcal{J}^{\mathbf 0}(X/S)$ is a smooth, separated group scheme of relative dimension $g$ over $S$ (see \cite[Chapter~9.4]{blr}) that can be characterised as the connected component containing the identity in the relative Picard variety of $X/S$.

By \cite{altman80} and \cite{esteves}, the functor that associates with every $S$-scheme $T$
\[
\Simp d{X/S}(T)=
\left\{
\begin{array}{c}\text{families  of rank $1$ torsion-free simple sheaves }\\ \text{of fibrewise degree $d$ parametrised by $T$  over $X/S$}\end{array}
\right\}
\]
is represented by an algebraic space $\operatorname{Simp}^{d}(X/S)$ with the following properties:
\begin{enumerate} 
\item\label{first} it is flat of relative dimension $g$, and  locally of finite type over $S$, 
\item it satisfies the existence part of the valuative criterion of properness,
\item it has connected and schematic fibres over the closed points,
\item\label{last} it admits an action of the generalised Jacobian $\mathcal{J}^{\mathbf 0}(X/S)$ which is free when restricted to the locus  $\pic^d(X/S) \subseteq \Simp d{X/S}$ parametrising line bundles. 
\end{enumerate}

The moduli functor has a generic $\mathbb{G}_m$-stabiliser and the algebraic space $\operatorname{Simp}^{d}(X/S)$ is obtained after rigidification of this stabiliser along one of the existing sections of the family.
Note that %
$\operatorname{Simp}^{d}(X/S)$ may fail to be separated over $S$ and it may fail to be of finite type (so it is not necessarily universally closed over $S$).

For each $g,n\geq 1$, the generalised Jacobian and  the moduli space  $\operatorname{Simp}^{d}(X/S)$ are also defined in the case where $X/S$ is the universal family $\Cmb{g}{n} / \Mmb gn$ over the moduli stack of stable $n$-pointed curves of genus~$g$ (see \cite{melouniversal}). %
In this case $\operatorname{Simp}^{d}(\Cmb{g}{n} / \Mmb gn)$ exists as a nonsingular Deligne--Mumford stack that is  representable over $\Mmb gn$, and it satisfies the additional properties \eqref{first}--\eqref{last} listed above. %

We now introduce the main object of study of this paper.%
\begin{definition}\label{def:fcj}
  	Let $X/S$ be a family of nodal curves. A degree $d$ \fcj $\overline{\mathcal{J}}^d(X/S)$  of the family $X/S$ is an open subspace of $\operatorname{Simp}^d(X/S)$ that is proper over $S$ and such that for all closed points $s \in S$,  the fibre  $\overline{\mathcal{J}}^d(X/S)_s$ is connected.
  	
  	A degree $d$ fine compactified universal  Jacobian $\Jmb dgn$ is a nonempty
		open substack of   $\operatorname{Simp}^{d}(\Cmb{g}{n}/\Mmb{g}{n})$ that is proper over $\Mmb{g}{n}$. 
\end{definition}
The compactified Jacobians in Definition~\ref{def:fcj} are called ``fine'' due to the existence of a tautological (or Poincar\'e) sheaf on them.

The present paper deals with \fcujs in the case $g=1$ (and $n\geq 1$).
The problem of classifying families of fine compactified universal Jacobians for nodal curves of arbitrary genus will be addressed in \cite{kasspa2}. 

Finally, let us also recall that \cite{meravi} gives a different definition of \fcj{} %
of a curve $X$. Their definition is only given in the case of compactified Jacobians obtained from some polarisation, but for a wider class of singular curves. When $X$ is a nodal curve, their definition yields a \fcj  in the sense of the present paper, see Section~\ref{notallpolar}.

\begin{remark}\label{numberconncomp}

It follows from the definition that a \fcuj has connected fibres over the closed points. This is why in Definition~\ref{def:fcj} we did not require this property in the universal case. 

In order to prove this, let us first point out that every \fcuj{}~$\Jmb dgn$ contains the degree~$d$ universal Jacobian~$\mathcal{J}^d_{g,n}=\Simp d{\mathcal{C}_{g,n}/\mathcal{M}_{g,n}}$. This follows from the fact that the restriction~$\Jmb dgn|_{\Mm gn}$ is nonempty for dimensional reasons, hence $\Jmb dgn|_{\Mm gn}$ is dense in~$\mathcal{J}^d_{g,n}$ because  the latter is irreducible.  As both spaces are proper over~$\Mm gn$, we deduce from density that the inclusion~$\Jmb dgn|_{\Mm gn} \subseteq \mathcal{J}^d_{g,n}$ is an equality.
 This also implies that the morphism~$\Jmb dgn \rightarrow \Mmb gn$ is surjective, as its image must contain the open dense subset~$\Mm gn$. 

To conclude the proof of connectedness of the fibres, let us recall from \cite[Theorem~4.17.(iii)]{delignemumford} that the number of connected components of the geometric fibres of a proper and flat morphism is a lower semicontinuous function. Then the claim follows from the fact that the fibres over points of~$\Mm gn$ are irreducible.
\end{remark}

	\begin{remark} \label{generalprop}
The pull-back of a degree $d$ \fcuj under a morphism $f \colon S \to \Mmb{g}{n}$ is a degree $d$ \fcj for the family $f^*\Cmb{g}{n}/S$. 
Indeed, openness, properness and having connected fibres over closed points are stable properties under base change. 
	\end{remark}
However, in general not all \fcjs of a family of nodal curves $X/S$ arise by pull-back from the universal family. For instance, a single nodal curve $X$ may have 
\fcjs that do not extend to any family with a smooth generic fibre, let alone to the whole universal family over $\Mmb gn$ (for some choice of $n$ marked points on $X$). We will give explicit examples of this phenomenon for $X$ a nodal curve of genus~$1$ in Section~\ref{genus1isolated}. This leads us to introduce the following crucial definition. 

\begin{definition}
A degree $d$ fine compactified Jacobian $\overline{\mathcal{J}}^d(X)$ of a nodal curve $X$ is \emph{smoothable} if there exists a smoothing
 $\mathcal{X}/T$ of $X$ over $T=\spec  \C [\![t]\!]$, i.e. a nonsingular scheme $\mathcal{X}$ over $T$ 
 such that the central fibre $\mathcal{X}_0$ equals $X$, together with a degree $d$ fine compactified Jacobian $\overline{\mathcal{J}}^d(\mathcal{X}/T)$ of the family $\mathcal{X}/T$ with central fibre $\overline{\mathcal{J}}^d(\mathcal{X}/T)_0$ equal to $\overline{\mathcal{J}}^d(X)$. \label{smoothable}
\end{definition}

It is well known that \fcjs exist. However, prior to this paper
all  examples of \fcjs in the literature were smoothable, and 
they were
obtained as moduli spaces of rank $1$ torsion-free sheaves that are semistable with respect to some polarisation. 
We will discuss their construction in the next subsection.

\subsection{Fine compactified Jacobians arising from  polarisations} 

\label{notallpolar}

There is a vast literature on compactified Jacobians arising from polarisations,
see for example \cite{oda79}, \cite{esteves} and \cite{simpson}. 
Although the meaning of the word ``polarisation'' varies in the different sources, these constructions can be adapted to produce the same fine compactified Jacobians. 
We will then see in Section~\ref{fcuj-classification} that there exist fine compactified universal Jacobians that cannot be obtained via any of these polarisations.

We will use the polarisation language of \cite{oda79}, which was adapted to the universal case by Kass and the first named author in \cite{kp3}. %
Let us review their constructions. We start with the space of polarisations.

\begin{definition} Given $\Gamma \in {G}_{g,n}$ a stable $n$-marked graph of genus~$g$, we define the space of polarisations
\[V^d(\Gamma):=\left\{\phi\in \mathbb{R}^{ \operatorname{Vert}(\Gamma)} :\ \sum_{v \in \operatorname{Vert}(\Gamma)} \phi(v) = d\right\} \subset \mathbb{R}^{\operatorname{Vert}(\Gamma)}.\] 

Then every morphism $f\colon\Gamma\rightarrow \Gamma'$ of stable marked graphs induces a morphism $f_*\colon V^d(\Gamma)\rightarrow V^d(\Gamma')$ by setting
\begin{equation} \label{phicompatible} f_*\phi(v') = \sum \limits_{f(v)=v'} \phi(v)
\end{equation}
and we define the space of universal polarisations as 
$V^d_{g,n}:=\varinjlim_{\Gamma\in G_{g,n}} V^d(\Gamma)$, i.e. as the space of assignments $\left(\phi(\Gamma)\in V^d(\Gamma)\colon \ \Gamma \in G_{g,n}\right)$ that are compatible with all graph morphisms.
\label{spaceofpolar}
\end{definition}

We are now ready to define the polarised fine compactified Jacobians.

\begin{definition} \label{phistab}	
 Let $X$ be a nodal curve with dual graph $\Gamma(X)$.	Given a polarisation $\phi \in V^d(\Gamma(X))$, we say that a rank~$1$ torsion-free simple sheaf $F$  of degree $d$ on $X$ is \emph{$\phi$-(semi)stable} if
	\begin{equation}
		\left| \deg_{X_0}(F)- \sum \limits_{v \in \operatorname{Vert}(\Gamma(X_0))} \phi(v) + \frac{\delta_{X_0}(F)}{2} \right| \underset{(\leq)}{<}  \frac{\#(\Gamma(X_0) \cap \Gamma(\overline{X_0^{c}}))-\delta_{X_0}(F)}{2}  %
		\label{Eqn: SymDefOfStability}
	\end{equation}
holds	for all proper subcurves  $\emptyset \subsetneq X_0 \subsetneq X$. Here $\delta_{X_0}(F)$ is 
the number of points in $X_0 \cap \overline{X_0^c}$ where $F$ is singular.

	We say that $\phi \in V^d(\Gamma(X))$ is \emph{nondegenerate} if every $\phi$-semistable sheaf is $\phi$-stable.  For $\phi\in V^d(\Gamma(X))$ nondegenerate, we define $\overline{\mathcal{J}}^d_{\phi}(X)$ to be the subscheme of $\operatorname{Simp}^d(X)$ parametrising  $\phi$-stable sheaves.

		Given a universal polarisation $\phi \in V_{g,n}^d$, we say that a family of rank~$1$ torsion-free sheaves of degree $d$ on a family of stable curves is \emph{$\phi$-(semi)stable} if Equation~\eqref{Eqn: SymDefOfStability} holds on all  fibres. We say that $\phi \in V_{g,n}^d$ is \emph{nondegenerate} if for all $\Gamma \in {G}_{g,n}$, the $\Gamma$-component $\phi(\Gamma)$ is nondegenerate in $V^d(\Gamma)$.	Finally, for $\phi\in V_{g,n}^d$ nondegenerate, we define $\Jmb dgn(\phi)$ to be the substack  of $\operatorname{Simp}^{d}(\Cmb{g}{n} / \Mmb gn)$ parametrising  $\phi$-stable sheaves on families of stable curves. \label{polarised}
\end{definition}

(In \cite{oda79} and \cite{kp3} the moduli functors of compactified Jacobians  are defined more generally, without assuming $\phi$ nondegenerate, and without assuming that the sheaves are simple. However, stable sheaves are always simple, see for example \cite[Lemma~2.18]{meravi}). 

As observed in \cite[Remark~4.6]{kp3}, the fine compactified (universal) Jacobians produced by this construction are the same as those defined by Esteves and Melo \cite{esteves,melouniversal}. 

Next, we explain how the moduli stacks defined in Definition~\ref{polarised} are indeed fine compactified Jacobians. 

\begin{proposition}
Let $X$ be a nodal curve with dual graph $\Gamma(X)$. For every nondegenerate
$\phi\in V^d(\Gamma(X))$, the moduli scheme $\overline{\mathcal{J}}^d_{\phi}(X)$ is a smoothable degree $d$ fine compactified Jacobian.

For every nondegenerate $\phi\in V_{g,n}^d$, the moduli stack $\Jmb dgn(\phi)$ is a fine compactified universal Jacobian. \label{polarisedarefine}
\end{proposition}

\begin{proof} The fact that $\overline{\mathcal{J}}^d_{\phi}(X) \subset \operatorname{Simp}^d(X)$ is connected and proper follows from \cite[Theorem~12.14]{oda79}. 
It remains to prove that $\phi$-stability is an open condition in $\Simp d{X/S}$. This follows from \cite[Proposition~34]{esteves}. However, applying that result requires us to translate our notion of $\phi$-stability to Esteves' stability with respect to a vector bundle $E$ on $X$, as defined in loc.cit. For the sake of completeness, we explain how this can be achieved.

Without loss of generality we may assume that $\phi(v)$ is a rational number for every $v\in\Vrt(\Gamma(X))$. Let $e \in \mathbb{N}$ be such that \[e \cdot \left(\phi(v)-\frac{d}{2g-2}(2g(v)-2)\right)=:a_v \in \mathbb{Z}\] for all $v \in \Vrt(\Gamma(X))$. 

Then one can check that our notion of stability with respect to $\phi$ corresponds to Esteves' $E$-stability for $E$ the  vector bundle of rank $e(2g-2)$ on $X$ defined by
\begin{equation} \label{conversion}
    E=\omega_X^{\otimes e(d+1-g)}\otimes \mathcal{O}\left(\sum_{v \in \Vrt(\Gamma(X))} a_v p_v \right)^{\otimes (2g-2)} \oplus \mathcal{O}_X^{\oplus ((2g-2)e)-1)},
\end{equation}
where $\omega_X$ is the dualising sheaf of $X$ and $p_v$ is the choice of one nonsingular point of the component of $X$ corresponding to $v$, for all $v \in \Vrt(\Gamma(X))$. We deduce that the moduli space  $\overline{\mathcal{J}}^d_{\phi}(X)$ equals Esteves' moduli space of $E$-stable sheaves ${J}_E^s(X)$.

Smoothability also follows from Esteves' construction. Indeed, let $\mathcal{X}/T$ be a smoothing of $X$ over $T= \spec \C [\![t]\!]$ as in Definition~\ref{smoothable}. Then each point $p_v$ on $X$ can be extended to a smooth section $\sigma_v \colon T \to \mathcal{X}$, and $\omega_X$ extends to the relative dualising sheaf $\omega_{\mathcal{X}/T}$ (and the trivial line bundle $\mathcal{O}_X$ can be extended to $\mathcal{O}_{\mathcal{X}}$). The line bundle $E$ on $X$ can therefore be extended to a line bundle $\mathcal{E}$ on $\mathcal{X}/T$ by obvious adaptation of Formula~\eqref{conversion}. By \cite[Theorem~A, Proposition~34]{esteves} Esteves' moduli space ${J}^s_{\mathcal{E}}(\mathcal{X}/T)$ of $\mathcal{E}$-stable sheaves on $\mathcal{X}/T$ is a degree $d$ \fcj for the family $\mathcal{X}/T$, and its central fibre equals ${J}_E^s(X)$.

Finally, properness and nonemptiness of $\Jmb dgn(\phi)$ follow from \cite[Corollary~4.4]{kp3}. Openness can be deduced again from Esteves' result as explained above for the case of a single curve, see \cite[Remark~4.6]{kp3}.
\end{proof}

We close this section by recalling the explicit wall and chamber description of the stability space $V^d(\Gamma(X))$ for $X$ a nodal curve of genus $1$ and of the stability space $V_{1,n}^d$ for all $d \in \mathbb{Z}$ and all $n\geq 1$. %
\begin{remark} \label{comparenodal}
The stability space for a nodal curve $X$ of arithmetic genus $1$ was studied in \cite[Section~7]{meravi}. Let us recall their results. 

Let us denote by $l$ (resp. by $r$) the number of irreducible components of $X$ contained (resp. not contained) in a rational tail of $X$. Each ordering of the components of $X$ induces an isomorphism $V^d(X) \cong \mathbb{R}^{r-1}\times \mathbb{R}^l$ (the stability assignment on the last of the $r$ components not contained in a rational tail equals $d$ minus the sum of the other assignments). 
Denoting by $(x_1, \ldots, x_{r-1},y_1, \ldots, y_{l})$ the coordinates in $\mathbb{R}^{r-1}\times\mathbb{R}^{l}$,  by \cite[Proposition~6.6, Proposition~7.4]{meravi}  the degenerate locus is the union of the hyperplanes $\{\sum_{i=r_1}^{r_2} x_i =c\colon c\in \Z\}$  over all $1 \leq r_1 \leq r_2 \leq r-1$, and of the hyperplanes
$\left\{y_k= t + \frac 12\colon t\in \Z\right\}$ over all  $1\leq k \leq l$.

From this description it can be deduced that, modulo translation by some line bundle on $X$, 
 the nondegenerate locus consists of $(r-1)!$ connected components inside the $(r-1)$-dimensional unit hypercube in $\mathbb{R}^{r-1}$, and hence that modulo translation there are $(r-1)!$ different degree $d$ fine compactified Jacobians of $X$.
Exactly one of these Jacobians contains the image of the Abel map $X \to \Simp d X$ associated with some (equivalently all) line bundles $M$ of degree $d+1$ on $X$, i.e. the map  $p \mapsto M \otimes I(p)$ for $I(p)$ the ideal sheaf of a moving point $p \in X$.
\end{remark}

In order to describe the stability space $V_{1,n}^d$, we will adopt the following conventions. 
if $\Gamma \in V^d_{1,n}$ is a graph with $2$ vertices, we  fix the coordinate on the affine subspace $V^d(\Gamma)\cong\R$ given by identifying any $\psi\in V^d(\Gamma)$ with the value $\psi(v)$ at the vertex $v$ which is \emph{not} incident to the $n$th half-edge.

For all $1 \leq i \leq n$, let $\Gamma_i \in G_{1,n}$ be the %
graph with two genus $0$ vertices joined by two edges, with the $i$-th marked point on one component and all other markings on the other component. For all $I \subseteq \{1,2,\dots,n\}$ with $|I|\geq 2$, let $\Gamma(I) \in G_{1,n}$ be the loopless 
graph with two vertices of genus $1$ and $0$ respectively, joined by one edge, with 
all the marked points in the set $I$ on the genus $0$ component 
and the marked points in the complement set $I^c$ on the genus $1$ component. 
\begin{proposition}\label{polytopes}
The projection 
\[\begin{array}{rcl}V^d_{1,n} &\longrightarrow& \prod_{i=1}^{n-1} V^d(\Gamma_i)\times \prod_{\begin{subarray}{c}I\subsetneq\{1,\dots,n\}\\|I|\geq 2\end{subarray}}V^d(\Gamma(I))\cong \mathbb{R}^{n-1} \times \mathbb{R}^{2^n-n-1}\\
\phi &\longmapsto &\left((\phi(\Gamma_i))_{1\leq i \leq n-1}, (\phi(\Gamma(I)))_{I \subseteq \{1,\dots,n\}, \ |I|\geq 2} \right)
\end{array}
\] %
is an isomorphism. %

Denoting by $x=(x_1, \ldots, x_{n-1})$ the coordinates in $\mathbb{R}^{n-1}$ and by $y=(y_1, \ldots, y_{2^n-n-1})$ the coordinates in $\mathbb{R}^{2^n-n-1}$,  the degenerate locus is the   union of the hyperplanes $\{\sum_{i\in J} x_i =c\colon c\in \Z\}$  over all nonempty $J\subseteq \{1,\dots,n-1\}$, and of the hyperplanes
$\left\{y_k= t + \frac 12\colon t\in \Z\right\}$ over all  $1\leq k \leq 2^n-n-1$.
\end{proposition}
\begin{proof}
This is the $g=1$ case of \cite[Corollary~3.6, Theorem~2]{kp3}.

\end{proof}

The following property, which is specific to the genus $g=1$ case, will be used in the proof of Lemma~\ref{existenceofj}.

\begin{corollary} \label{surjectivegenus1}
The natural projection map 
$V_{1,n}^d \to V^d(\Gamma)$
is surjective for every stable graph $\Gamma\in G_{1,n}$.
\end{corollary}
\begin{proof}
The claim can be easily checked by expressing the map in terms of the coordinates of Remark~\ref{comparenodal} for $V^d(\Gamma)\cong \mathbb R^{r-1+l}$ and the coordinates of Proposition~\ref{polytopes} for $V_{1,n}^d \cong \mathbb{R}^{n-1} \times \mathbb{R}^{2^n-n-1}$. 
\end{proof}

\subsection{Some general results on fine compactified Jacobians of families} Here we collect some general results on fine compactified Jacobians that will be needed later. 

Our first observation is that in the case where the fibres are irreducible curves, the moduli space of simple sheaves is already proper, so the only fine compactified Jacobian is the entire moduli space of simple sheaves.
\begin{lemma} \label{irreducible}
  If $X/S$ is a family of \emph{irreducible} nodal curves, then $\operatorname{Simp}^d(X/S)$ is the only fine compactified Jacobian for the family $X/S$. 
\end{lemma}
\begin{proof}
 From \cite[Proposition~34]{esteves} and  \cite[Section~1.2]{esteves} we deduce that the scheme $\operatorname{Simp}^d(X/S)$ is proper and connected when $X$ is irreducible. The result follows then from the fact that fine compactified Jacobians are by definition nonempty, open and proper subspaces of $\operatorname{Simp}^d(X/S)$.
\end{proof}

In the genus~$1$ case, Lemma~\ref{irreducible} has the following consequence.
 \begin{lemma} \label{irred1}
   Let $X/S$ be a family of \emph{irreducible} nodal curves of arithmetic genus~$1$ that admits a smooth section $q$, and $\overline{\mathcal{J}}^d(X/S)$ a \fcj. Then the Abel map mapping a point $p$ over $s$ to the sheaf $I(p) \otimes \mathcal{O}_{X_s}((d+1)q(s))$ induces an isomorphism $X \to \overline{\mathcal{J}}^d(X/S)$.
 \end{lemma}
   The case when $X/S$ is smooth and with a section is a well-known classical fact, see e.g. \cite[Proposition 5.3.2]{cossec-dolgachev}.
  \begin{proof}
This follows from \cite[Ex.~(8.9, iii)]{altman80} in light  of the equality $\Simp d{X/S}=\overline{\mathcal{J}}^d(X/S)$ proved in Lemma~\ref{irreducible}.
  \end{proof}

We are now going to prove a series of general results  for fine compactified Jacobians of curves of any genus. The first is that forming fine compactified Jacobians commutes with products.
\begin{lemma} \label{product}
  Let $X/S$ be a family of nodal curves, and let $T$ be a scheme. The operation of restricting sheaves to a fibre of the product identifies each degree~$d$ fine compactified relative Jacobian $\overline{\mathcal{J}}^d((X\times T)/(S\times T))$ with $\overline{\mathcal{J}}^d(X/S) \times T $ for some degree~$d$ fine compactified Jacobian $\overline{\mathcal{J}}^d(X/S)$ of $X/S$. %
\end{lemma}
\begin{proof}
By \cite[Section~4]{esteves}
we may identify $\operatorname{Simp}^d((X\times T)/(S\times T))=\operatorname{Simp}^d(X/S)\times T$. The claim follows because the map $S\times T \to \Mmb gn$ factors through the projection to $S$, namely $S\times T \to S \to \Mmb gn$.
\end{proof}

In the following result we prove that the fine compactified Jacobian of a curve with a separating node is a product of fine compactified Jacobians of the two subcurves. %

\begin{lemma}\label{separatingnode}
Let $\xymatrix{{\mathcal X_1} \ar[r]& {B_1} \ar@/_1pc/[l]_{\sigma_1}}$ and $\xymatrix{{\mathcal X_2} \ar[r]& {B_2} \ar@/_1pc/[l]_{\sigma_2}}$ be 
two families of nodal curves with distinguished sections $\sigma_i$ factorising through the smooth locus of $\mathcal X_i$ for $i=1,2$. 

Let us denote by $\mathcal X\rightarrow B:=B_1\times B_2$ the family of curves obtained by gluing the fibres of $\mathcal X_1$ and $\mathcal X_2$ transversely along the smooth sections $\sigma_1$ and $\sigma_2$. Then for any degree $d$ \fcj $\overline{\mathcal J}^d(\mathcal X/B)$ of $\mathcal X/B$, there exists a unique partition $d_1+d_2=d$ and a unique pair of fine compactified Jacobians $\overline{\mathcal{J}}^{d_i}(\mathcal X_i/B_i)$ for $i=1,2$ such that restriction to the subcurves induces an isomorphism
 \[ \overline{\mathcal{J}}^d(\mathcal X/B)\cong \overline{\mathcal{J}}^{d_1}(\mathcal X_1/B_1) \times \overline{\mathcal{J}}^{d_2}(\mathcal X_2/B_2).\]
\end{lemma}

\begin{proof}
The fibres $X_b$ of $\mathcal X/B$ consist of two curves glued together along a separating node. 
Note that the restriction of a torsion-free sheaf to an irreducible component fails to be itself torsion-free if the original sheaf is singular at some of the intersection points of the component and the closure of its complement.

If $F$ is a rank~$1$ simple sheaf on a curve $X$, then its stalk $F_p$ at a separating node $p$ is locally free. We deduce that restricting to $(\mathcal X_1\times B_2)/B$ and $(B_1\times \mathcal X_2)/B$ induces an isomorphism
\[\begin{array}{rl}
\operatorname{Simp}^{d}(\mathcal X/B) &\cong \coprod_{d_1+d_2=d} \operatorname{Simp}^{d_1}(\mathcal X_1/B_1) \times \operatorname{Simp}^{d_2}(\mathcal X_2/B_2).
\end{array}
\]
Our statement then follows from the fact that the Euler characteristic of the restriction of the sheaves in a (flat) family of sheaves to the family of subcurves $\mathcal X_i$ is locally constant for $i=1,2$. 
Hence the total degree of the restriction of the sheaves is also locally constant. 
The fact that the image $\overline{\mathcal{J}}^{d_j}(\mathcal X_j/B_j)$ of the projection $\overline{\mathcal{J}}^d(\mathcal X/B)\rightarrow \Simp{d_1}{\mathcal X_j/B_j}$ arises as the pull-back of a \fcuj follows from the fact that the gluing map at a separating node is defined at the level of moduli spaces of curves.
\end{proof}
 
 \subsection{Some low genus results}
In this section we recall some basic results on \fcjs of curves of low genus.

If $X$ has arithmetic genus $0$, then it is a tree of rational curves, hence its fine compactified Jacobians are easy to describe from Lemma~\ref{separatingnode}.

\begin{corollary} \label{genus0} If $X$ is a nodal curve of arithmetic genus $0$, every fine compactified Jacobian $\overline{\mathcal J}^d(X)$ is isomorphic to a point.
  \end{corollary}
  \begin{proof}
  All nodes of a curve of arithmetic genus $0$ are separating, so its dual graph is a tree, and its irreducible components are nonsingular and of geometric genus $0$.  The result follows by applying Lemma~\ref{separatingnode} with $B_1=B_2=\operatorname{pt}$, and by reasoning inductively on each irreducible component starting from the leaves.
  \end{proof}
 
Another consequence of Lemma~\ref{separatingnode} is that the existence of rational tails does not contribute to the geometry of a fine compactified Jacobian. %

\begin{corollary}\label{forgettails} Assume $X= Z \cup (Y_1 \cup \ldots \cup Y_k)$ is a nodal curve such that the $Y_i$ are irreducible components of rational tails of $X$. Let $\overline{\mathcal J}^d(X)$ be a fine compactified Jacobian of $X$, and let $d_i$ be the degree of the restriction of an element of $\overline{\mathcal J}^d(X)$ to $Y_i$.  For $d'=d-\sum d_i$, there exists a unique fine compactified Jacobian $\overline{\mathcal J}^{d'}(Z)$ such that restricting sheaves to $Z$ induces an isomorphism \[ \overline{\mathcal J}^d(X) \to \overline{\mathcal J}^{d'}(Z).\]
In particular, all elements of $\overline{\mathcal J}^d(X)$ have the same degree $d_i$ when restricted to $Y_i$.
\end{corollary}
\begin{proof}
Similar to \ref{genus0} and omitted.
\end{proof}

We conclude by recalling some specific properties of the moduli space of simple sheaves in the genus $1$ case.
\begin{proposition}\label{simple:I_n}
Let $X$ be a nodal curve of genus $1$, %
then
\begin{enumerate} 
\item If $X$ has at least one nonseparating node, the generalised Jacobian $\mathcal J^{\mathbf 0}(X)$ is isomorphic to the multiplicative group $\mathbb{G}_m$.
\item The nonsingular locus of $\Simp dX$ is the locus of line bundles. If two line bundles $L_1, L_2$ belong to the same irreducible component of $\Simp dX$, then we have 
$\deg L_1|_{Y} = \deg L_2|_{Y}$ for each irreducible component $Y$ of $X$.
\item\label{atmost1node} The singular points of $\operatorname{Simp}^d(X)$ are nodes. They parametrise sheaves that are singular at exactly one nonseparating node of $X$. 
\end{enumerate}
\end{proposition}

\begin{proof}
The first point is standard and follows from the fact that $X$ is connected, with arithmetic genus $1$ more than the sum of the geometric genera of all its components. (See \cite[p.~90]{acg2}).

The fact that the singularities are at worst nodes and that the singular points correspond to the noninvertible sheaves is a consequence of \cite[Theorem~5.10]{cmkv1}. 
 
The other assertions follow then from analysing the action of the generalised Jacobian $\mathcal J^{\mathbf 0}(X)$ (introduced in the beginning of Section~\ref{generalfine}), in particular from the fact that the action of $\mathcal J^{\mathbf 0}(X)$ respects the decomposition into the nonsingular and the singular locus of $\Simp dX$ and that the $\mathcal J^{\mathbf 0}(X)$-orbit of a line bundle $L$ consists of all line bundles with the same multidegree as $L$.

The second part of \eqref{atmost1node} follows from the assumption that $F$ is simple. That implies that the singular points of $F$ do not disconnect the curve $X$ (the automorphism group of $F$ contains an algebraic torus of dimension equal to the number of connected components of $X$ minus the singular points of $F$). Since we are assuming that $X$ has arithmetic genus $1$, a nonseparating set of nodes can contain at most $1$ element. 
\end{proof}

\section{Fine compactified Jacobians of necklace curves}
\label{genus1isolated}

This section is devoted to the study of \fcjs of nodal curves of genus $1$ without rational tails. 
By Corollary~\ref{forgettails} this is enough to describe the \fcjs of all nodal curves of genus $1$. It will be convenient to have the following concise name for the curves studied in this section.

\begin{definition} \label{necklace}
    We say that a genus~$1$ stable pointed curve is a \emph{necklace curve} if it cannot be disconnected by resolving one of its nodes.
\end{definition}

\begin{remark}
Necklace curves appear in the Kodaira classification of fibres of elliptic fibrations as curves of type $I_k$. In particular, nonsingular curves of genus~$1$, i.e. curves of Kodaira type $I_0$, are considered as necklace curves.
\end{remark}

Our aim is to give a classification of the \fcjs of a fixed necklace curve $X$.  
Our main technical tool is Lemma~\ref{description-I_n}, where we give a combinatorial classification of all degree $d$ \fcjs of $X$. This will allow us to exhibit in Remark~\ref{strangeexamples} examples of fine compactified Jacobians that are \emph{not} smoothable (see Definition~\ref{smoothable}). Another consequence is Corollary~\ref{cor:fcjs}, where we show that the smoothable Jacobians of $X$ are, up to translation, in one-to-one correspondence with 
the cyclic orderings %
 of the set of nodes of $X$.

One important combinatorial characterisation of the smoothable fine compactified Jacobians of necklace curves, which we will use in 
Section~\ref{fcuj-classification} to classify all genus~$1$ \fcujs{}, is the following:

\begin{proposition}
\label{f-for-fcjs}
Let $X$ be a necklace curve with $n\geq 2$ components. 
Let 
\[\mathcal C_{n-1}:=\left\{I=\{r,r+1,\dots,s\}: 1\leq r\leq s\leq n-1\right\}\] 
be the set of sequences of consecutive integers between $1$ and $n-1$, and let
$f\colon \mathcal C_{n-1}\rightarrow \Z$ be an integer valued function
 that satisfies the mild superadditivity condition:
\begin{equation}\label{condfunction}
0\leq f(I\cup J) - f(I) - f(J) \leq 1
\end{equation}
whenever $I$ and $J$ are disjoint and all three sets $I,J,I\cup J$ belong to $\mathcal C_{n-1}$. 

Then for every  $d\in \mathbb{Z}$ there exists a unique \emph{smoothable} degree $d$ \fcj{} $\overline{\mathcal J}^d(X)$ such that 
\[\min\left\{\sum_{i\in I}d_i:\; \mathcal J^{(d_1,\dots,d_n)}(X)\subseteq \overline{\mathcal J}^d(X)\right\} = f(I)
\]
holds for all $I\in\mathcal C_{n-1}$.
\end{proposition}

The notation in this Proposition uses the following conventions for necklace curves and the multidegree of simple sheaves on them, which we will adopt for the rest of the section.
\begin{notation}
\begin{enumerate}
    \item \label{nota-i}
Let $X$ be a necklace curve with $n\geq 2$ components.
We choose an orientation of the dual graph of $X$ and order the $n$ components of $X$ as $\{C_i:\;i\in\Z/n\Z\}$ according to this orientation. In particular, for each index $i$ the components $C_i$ and $C_{i+1}$  intersect at $1$ point, which we denote by $P_i$. We will also choose a distinguished marked point for each component and call it $Q_i\in C_i\setminus\{P_{i-1},P_i\}$ for each $i\in\Z/n\Z$. 
\item For $\mathbf d \in \Z^n$, define   $\mathcal J^{\mathbf d}(X)$ of $X$ as the moduli space of line bundles of multidegree $\bf d$ on $X$. (The special case of $\mathcal J^{\mathbf 0}(X)$, known as the generalised Jacobian, has already been introduced at the beginning of Section~\ref{generalfine}). In this section, the components of the multidegree will always be ordered as in \eqref{nota-i}.
\item  Let $F\in \operatorname{Simp}(X)$ be noninvertible.
Then the sheaf $F$ is singular at a unique node $P_j$ of $X$ by  Proposition~\ref{simple:I_n}. 
 By work of Seshadri 
(see e.g. \cite[Lemma 1.5]{alexeev-cjtm})
the sheaf $F$ can be obtained as the direct image $F=f_*L'$ of an invertible sheaf $L'$ of total degree $d-1$ on the partial normalisation $f\colon\tilde X_j\rightarrow X$ of $X$ at $P_j$. Hence, if $F\in\Simp dX$ is noninvertible, 
we can associate with it a node $P_j\in X$ and the multidegree $\mathbf d' = (\deg_{\tilde C_1}(L'),\dots,\deg_{\tilde C_n}(L'))$ of $L'$, where $\tilde C_i$ denotes the component of $\tilde X_i$ mapping to $C_i$. 
We denote by $N_j^{\mathbf d'}$ the sheaf (unique up to isomorphism) that is singular at the node $P_j$ and that is the direct image of an invertible sheaf  of multidegree $\mathbf d'=(d'_1,\dots,d'_n)$ on $\tilde X_j$. 
\end{enumerate}
\end{notation}

By Proposition~\ref{simple:I_n}, the nonsingular locus of $\Simp dX$ is the disjoint union of the open subsets $\mathcal J^{\mathbf d}(X)$ for all $\mathbf d = (d_1,\dots,d_n)$ with $d_1+\dots+d_n=d$. The singular locus of $\Simp dX$ consists of nodes $N_j^{\mathbf d'}$ with $d'_1+\dots+d'_n=d-1$ that correspond to singular sheaves.

Since $X$ has genus~$1$, the moduli space $\operatorname{Simp}^d(X)$ is a $1$-dimensional scheme locally of finite type. For $n\geq 2$ %
it is neither separated nor of finite type. However, as we will see, it still contains an infinite number of different \fcjs{}.

\begin{remark} \label{bananaexample}
In order to gain some intuition, let us start by considering the case of sheaves of total degree $d=0$ on a necklace curve $X$ with $n=2$ components. 

Then $\Simp 0X$ is
the disjoint union of open subsets $\mathcal J^{(\alpha,-\alpha)}(X)\cong \mathbb G_m$ for all $\alpha\in\Z$ and points of the form $N_1^{(\beta,\beta-1)}$ or $N_2^{(\beta,\beta-1)}$ for all $\beta\in\Z$, representing sheaves singular at $P_1$ and $P_2$, respectively. To describe the closure of $\mathcal J^{(\alpha,-\alpha)}(X)$ in $\Simp 0X$, we need to describe which noninvertible sheaves arise as flat limits of invertible sheaves of multidegree $(\alpha,-\alpha)$: it is easy to see that they correspond exactly to the four points $N_1^{(\alpha-1,-\alpha)}$, $N_2^{(\alpha-1,-\alpha)}$, $N_1^{(\alpha,-\alpha-1)}$ and $N_2^{(\alpha,-\alpha-1)}$. This is summarised in Figure~\ref{figure-n=2}, where the horizontal lines represent the open subsets of the form $\mathcal J^{(\alpha,-\alpha)}(X)$ and the symbols $\bullet$ and $\mathbin{\vcenter{\hbox{\scalebox{0.5}{$\blacksquare$}}}}$ represent sheaves singular at $P_1$ and at $P_2$, respectively. 
\begin{figure}

\[\xymatrix@R=0.4em@C=1ex{
&&&& &&&&&\\
&&&& &&&&&\\
&{ \blacksquare}\ar@{.}[uu] &&&\ar@{.}[uu] &&& { \mathbin{\vcenter{\hbox{\scalebox{2}{$\bullet$}}}}}\ar@{.}[uu]\\
\ar@{-}[r]&\ar@{-}[rrrrrr]^{(2,-2)}&&& &&&\ar@{-}[r]&\\
(1,-2)&{ \mathbin{\vcenter{\hbox{\scalebox{2}{$\bullet$}}}}} &&& &&& { \blacksquare} &(1,-2)\\
\ar@{-}[r]&\ar@{-}[rrrrrr]^{(1,-1)}&&& &&&\ar@{-}[r]&&\\
(0,-1)&{ \blacksquare} &&& &&& { \mathbin{\vcenter{\hbox{\scalebox{2}{$\bullet$}}}}} &(0,-1)\\
\ar@{-}[r]&\ar@{-}[rrrrrr]^{(0,0)}&&& &&&\ar@{-}[r]&\\
(-1,0)&{ \mathbin{\vcenter{\hbox{\scalebox{2}{$\bullet$}}}}} &&& &&& { \blacksquare} &(-1,0)\\
\ar@{-}[r]&\ar@{-}[rrrrrr]^{(-1,1)}&&&\ar@{.}[ddd] &&&\ar@{-}[r]&\\
(-2,1)&{ \blacksquare}\ar@{.}[dd] &&& &&& { \mathbin{\vcenter{\hbox{\scalebox{2}{$\bullet$}}}}}\ar@{.}[dd] & (-2,1)\\
&&&& &&&&&\\
&& &&&&&&&\\
}\]

\caption{Structure of $\Simp 0X$ for $X$ a necklace curve with $n=2$ components. 
\label{figure-n=2}}
\end{figure}
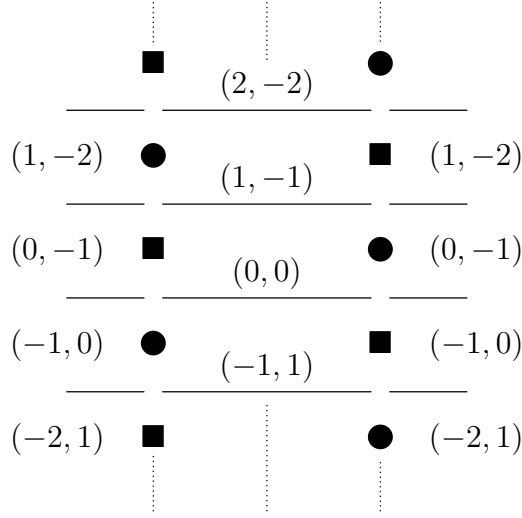

Using the symmetries of the curve $X$, one can check that the closure of $\mathcal J^{(\alpha,-\alpha)}(X)$ in $\Simp 0X$ is a nonseparated scheme isomorphic to $\Pp1$ with a double  origin and a double infinity. For this reason, the irreducible components of any \fcj{} $Y\subset\Simp 0X$ are copies of $\Pp1$ obtained by adding  to an open subset $\mathcal J^{(\alpha,-\alpha)}(X)$ exactly one boundary point over $0$ and one over $\infty$. From the description in Figure~\ref{figure-n=2} 
it then follows that $Y$ must be of the form 
\[Y = \mathcal J^{(\alpha,-\alpha)}(X) \cup \mathcal J^{(\alpha+1,\alpha-1)}(X) \cup \{N_1^{(\alpha,-\alpha-1)}, N_2^{(\alpha,-\alpha-1)}\}
\]
for some $\alpha\in\Z$.
\label{rem:n=2}\end{remark}

In the next Lemma  we generalise Remark~\ref{rem:n=2}, giving a combinatorial classification of all degree $d$ \fcjs{} of a necklace curve $X$ with an arbitrary number of components $n\geq 2$.%

\begin{lemma}\label{description-I_n}
Let $X$ be a necklace curve with $n\geq 2$ components.
If $Y\subseteq \Simp dX$ is a \fcj{},
then $Y$ is a necklace curve with $m=\rho n$ components for some $\rho \geq 1$. Furthermore, there is a sequence of $m$ integers $(j_{\ell}:\;\ell\in\Z/m\Z)$ satisfying 
\[\#\{\ell\in\Z/m\Z:\;j_\ell = k\}=\rho  \text{ for every }k=1,\dots,n\]
and a multidegree $\mathbf D\in \Z^m$ (depending only on $Y$ and $(j_\ell:\;\ell\in\Z/m\Z)$) such that 
\begin{equation}\label{description-Y}
    Y = \bigcup_{1\leq k\leq m}\mathcal J^{\mathbf D + \mathbf d_k}(X) \cup \{N_{j_1}^{\mathbf D + \mathbf d'_1},\dots,N_{j_m}^{\mathbf D + \mathbf d'_m}\}
\end{equation}
where we define the multidegrees $\mathbf d_k$ and $\mathbf d'_k$ as  
\[\begin{array}{ll}
\mathbf d_k&:= \sum_{1\leq \ell\leq k}(\mathbf e_{j_\ell}-\mathbf e_{j_\ell+1}),\\
\mathbf d'_k&:= \mathbf d_k-\mathbf e_{j_k},
\end{array}
\]
where $\mathbf e_i$ is the $i^{\text{th}}$ vector of the standard basis.

Conversely, every subscheme of the form \eqref{description-Y} is a \fcj of $X$, 
 provided that for every $1\leq \rho'\leq \rho -1$, all subsets $(j_{\ell+k}:\;k=1,\dots, \rho'n)$ consisting of $\rho'n$ consecutive indices do \emph{not} contain all indices $1,\dots,n$ with the same multiplicity~$\rho'$. 
The sequence $(j_{\ell}:\;\ell\in\Z/m\Z)$ is unique up to shift  
$(j_{\ell}:\;\ell\in\Z/m\Z)\sim (j_{\ell+k}:\;\ell\in\Z/m\Z)$ by some $k\in\Z/m\Z$.
\end{lemma}

We will later show that a \fcj $Y$ of a necklace curve $X$ is smoothable if and only if $n=m$ holds in the description of Lemma~\ref{description-I_n}, i.e. $Y$ has the same number of components as $X$.

\begin{proof}%

We start by constructing an affine open cover of the scheme $\Simp dX$. 
If $L\in\Simp dX$ is an invertible sheaf of multidegree $\mathbf d$, then  it is contained in an open subset of the form $\mathcal J^{\mathbf d}(X)\cong \mathbb{G}_m$. It remains to describe a neighbourhood of the nodes $N_j^{\mathbf d'}$ of $\Simp dX$. To identify the two branches of the node $N_j^{\mathbf d'}$, we observe that
the sheaves of the form $N_j^{\mathbf d'}$ can be obtained as flat limits both of invertible sheaves of multidegree $\mathbf d'+\mathbf e_j = (d'_1,\ldots,d'_j+1,d'_{j+1},\ldots,d'_n)$ and of invertible sheaves of multidegree $\mathbf d'+\mathbf e_{j+1} = (d'_1,\ldots,d'_j,d'_{j+1}+1,\ldots,d'_n)$. This shows that there is an open neighbourhood of $N_j^{\mathbf d'}$ in $\Simp dX$ which is isomorphic
to $\{xy=0\} \subset \mathbf A^2$, where the line $\{x=0\}\setminus \{(0,0)\}$ corresponds to $\mathcal J^{\mathbf d'+\mathbf e_j}(X)$ and the line $\{y=0\}\setminus \{(0,0)\}$ corresponds to $\mathcal J^{\mathbf d'+\mathbf e_{j+1}}(X)$.

Next, we consider the question of how to glue together the open neighbourhoods of the singular points $N_j^{\mathbf d'}$ along the nonsingular open subsets $\mathcal J^{\mathbf d}(X)$ of $\Simp dX$ to produce a \emph{proper} subscheme of $\Simp dX$.
 Without loss of generality, it is enough to understand how to glue together two open neighbourhoods of nodes of $\Simp 0X$ along the generalised Jacobian $\mathcal J^{\mathbf 0}(X)$ (the other cases can be obtained from this after tensoring with some invertible sheaf on $X$).
By the discussion above, the Zariski closure of the generalised Jacobian $\mathcal J^{\mathbf 0}(X)$ in $\Simp 0X$ contains a node $N_j^{\mathbf d'}$ of $\Simp 0X$ if and only if we have either $\mathbf d' = -\mathbf e_j$ (type 1) or $\mathbf d' = -\mathbf e_{j+1}$ (type 2). 
Figure~\ref{fig:twotypes-new} exemplifies this in the case $n=3$: the type 1 boundary points are represented on the left hand side and the type 2 ones on the right hand side.

\begin{figure}
    \centering
{\scalebox{0.8}{\includegraphics{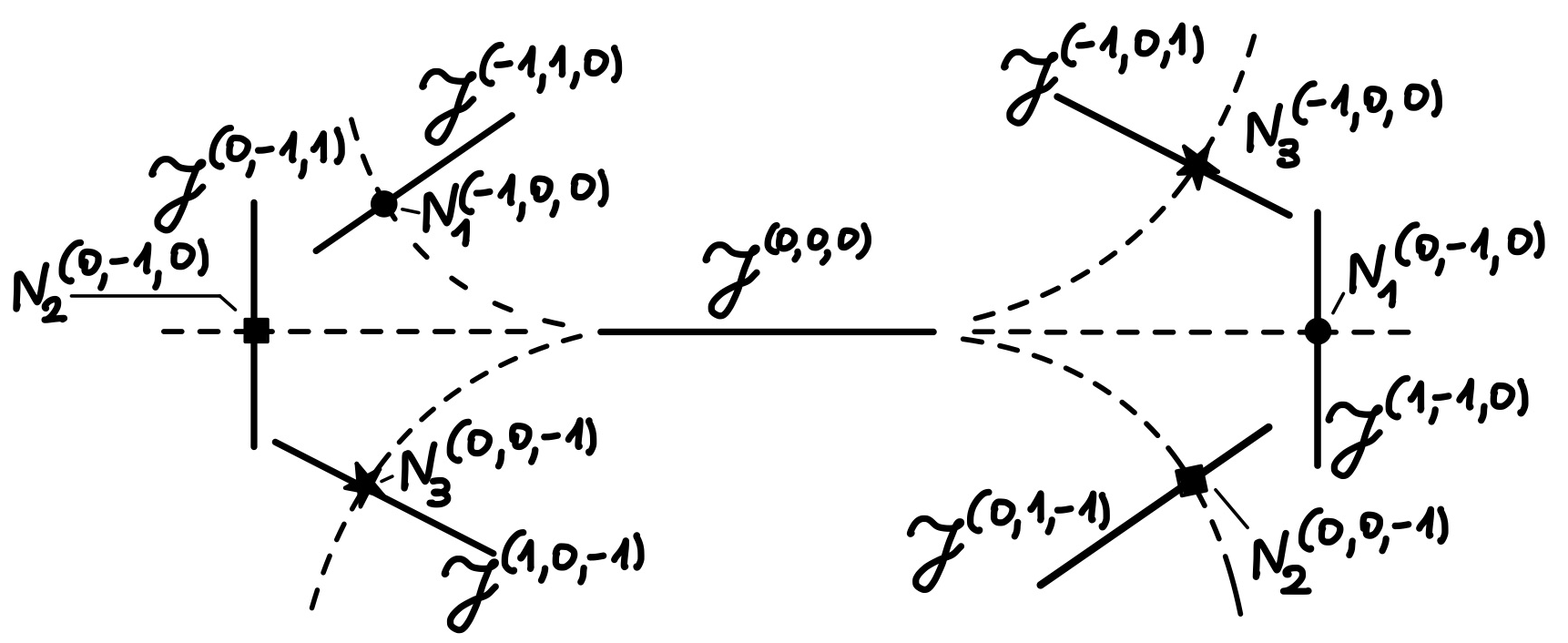}}}

          \caption{Boundary points of $\mathcal J^{\mathbf 0}(X)$ in $\Simp 0X$ in the case $n=3$. The type 1 boundary points are sketched on the left hand side, the type 2 ones are on the right hand side.
    \label{fig:twotypes-new}}
\end{figure}

Note that the boundary of $\mathcal J^{\mathbf 0}(X)$ in $\Simp 0X$ contains points representing sheaves $F$ that may be singular at any of the nodes $P_i$ of $X$.
This is consistent with the action of the cyclic group $\Z/n\Z$ on $\Simp 0X$ induced by the $\Z/n\Z$-action on $X$ generated by $\gamma\in\operatorname{Aut}(X)$ with $\gamma(P_i)=P_{i+1}, \gamma(Q_i)=Q_{i+1}$ for all $i\in\Z/n\Z$. It is easy to check that $\gamma$ acts trivially on the generalised Jacobian $\mathcal J^{\mathbf 0}(X)$. On the other hand, the orbit  of the node $N_j^{-\mathbf e_j}\in\Simp 0X$ under the subgroup generated by $\gamma$ is $\{N_i^{-\mathbf e_i}:\; i\in\Z/n\Z\}$ and the orbit of $N_j^{-\mathbf e_{j+1}}$ is $\{N_i^{-\mathbf e_{i+1}}:\; i\in\Z/n\Z\}$. 
Since all boundary points of type 1 lie in the same $\Z/n\Z$-orbit, the union of $\mathcal J^{\mathbf 0}(X)$ and $\{N_i^{-\mathbf e_i}:\; i\in\Z/n\Z\}$ is nonseparated and isomorphic to an affine line with $n$ origins. The same holds for the union of $\mathcal J^{\mathbf 0}(X)$ and the set $\{N_i^{-\mathbf e_i}:\; i\in\Z/n\Z\}$ of boundary points of type 2.

However, the two types of boundary points, as well as their affine neighbourhoods, are interchanged by any automorphism $\sigma$ of $X$ that interchanges $C_j$ and $C_{j+1}$, such as the one defined by $\sigma(P_i)=P_{2j+1-i}$, $\sigma(Q_i)=Q_{2j-i}$ for all $i\in\Z/n\Z$. Since this automorphism does not respect the cyclic orientation of $X$, we have that %
$\sigma$ acts on $\mathcal J^{\mathbf 0}(X)$ as $\sigma(L) = L^{-1}$. This means that if we want to obtain a separated scheme, we are allowed to glue the neighbourhoods of two nodes of $\Simp 0X$  along $\mathcal J^{\mathbf 0}(X)$ if and only if they are of different types, as shown in Figure~\ref{glued}. In particular, if we consider the curve in Figure~\ref{glued} as oriented counterclockwise, every time we pass from a component $\mathcal J^{\mathbf d}(X)$ to the next one $\mathcal J^{\mathbf d'}(X)$, 
we have to pass through a point that corresponds to a sheaf that is singular at a node $P_k$ of $X$, and 
the multidegree of the components changes with the formula $\mathbf d' = \mathbf d + \mathbf e_k - \mathbf e_{k+1}$. 

\begin{figure}
\[
\xymatrix@=1.5em{
&&&&&&\\
&&&&&&\\
&&&&&&\\
&*=0{\bullet} \ar@{-}[uuu]^{\mathcal J^{\mathbf e_{i+1}-\mathbf e_i}(X)}\ar@{-}[rrr]_{\mathcal J^{\mathbf 0}(X)}\ar@{..}[ld]\ar@{-}[d]\ar@{-}[l]&&&
*=0{\bullet} \ar@{-}[uuu]_{\mathcal J^{\mathbf e_{j}-\mathbf e_{j+1}}(X)}\ar@{..}[rd]\ar@{-}[d]\ar@{-}[r]&\\
{N_i^{-\mathbf e_i}}&&&&&{N_j^{-\mathbf e_{j+1}}}  }
\]
\caption{Gluing along $\mathcal J^{\mathbf 0}(X)$. \label{glued}}
\end{figure}
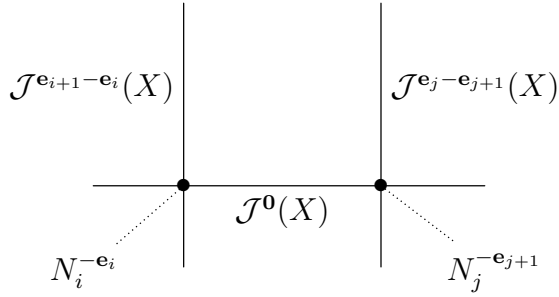

All \fcjs{} $Y\subseteq\Simp dX$ %
will result from a sequence of $m$ such gluings. %
 In particular, the components of $Y$ will be copies of $\Pp1$ meeting at nodes of $\Simp dX$, i.e. the subscheme $Y$ is isomorphic to a necklace curve with $m$ nodes.
By the discussion above, we can always choose a cyclic ordering of the components $Y_1,\dots,Y_m$ of $Y$ (i.e. an orientation of the dual graph of $Y$) such that, if we denote by $\mathbf d_\ell$ the multidegree of a sheaf corresponding to a general point in $Y_\ell$, we have 
\begin{equation}\label{ind}\mathbf d_{\ell+1} = \mathbf d_{\ell}+\mathbf e_{j_\ell}-\mathbf e_{j_\ell+1} \text{ for all }\ell\in\Z/m\Z,\end{equation}
where $P_{j_\ell}$ is the singular point
 of the sheaf corresponding to the intersection $\{F_\ell\} = Y_{\ell-1}\cap Y_{\ell}$. 
Using that $\mathbf d_1=\mathbf d_{m+1}$, we obtain

\begin{equation}\label{eq:zerosum}
\sum_{\ell\in\Z/m\Z}(\mathbf e_{j_\ell}-\mathbf e_{j_\ell+1}) = \mathbf 0.
\end{equation}
Since \eqref{eq:zerosum} is equivalent to the $m$ equalities 
\[\#\{\ell:\;j_\ell =i\} - \#\{\ell:\;j_{\ell+1} =i\}=0,\ \ \ i\in\Z/m\Z\]
we have that all indices $i\in\Z/n\Z$ appear the same number of times $\rho\geq 1$ in the sequence $(j_1,\ldots,j_m)$.
The sequence $(j_\ell:\;\ell\in\Z/m\Z)$ is uniquely determined by $Y$ up to cyclic shift, since it represents the sequence of nodes at which the sheaves in $Y$ are singular. %

At this point, we are ready to prove that $Y$ can be expressed as in the right hand side of Formula~\eqref{description-Y}. Let $F\in Y$ be a sheaf which is singular %
 at the point $P_{j_1}$. Without loss of generality, we may assume that the component preceding $F$ in the cyclic order is $Y_m$ and then set $\mathbf D:=\mathbf d_m$ to be the multidegree of $Y_m$. After taking the tensor product of $Y$ with an invertible sheaf of multidegree $- \mathbf D$, we may assume $d=0$ and $\mathbf D = \mathbf 0$. Then \eqref{ind} allows us to define the general multidegrees $\mathbf d_k$ of the components of $Y$ recursively. By the discussion on the affine charts of $\Simp dX$, the union defined on the right hand side of \eqref{description-Y} is an open substack. 
  This union is proper if and only if no multidegree $\mathbf d$ of a component $\mathcal J^{\mathbf d}(X)$ is repeated more than once. This is equivalent to requiring that no sequence of $0<m'<m$ consecutive indices in $(j_\ell:\;\ell\in\Z/m\Z)$ defines a necklace curve.
\end{proof}

\begin{remark} \label{strangeexamples}
Lemma~\ref{description-I_n} gives a combinatorial construction of all connected, open and proper 
subschemes $Y$ of $\Simp dX$. 
In Proposition~\ref{prop:I_n} below we will show that any \fcj of a necklace curve has the same number of components as the necklace curve itself.
In particular, the case $\rho\geq 2$ in Lemma~\ref{description-I_n}
gives rise to \fcjs of $X$ that are not smoothable. 

We give below an explicit example with $n=3$ and $\rho=2$, in which the sequence of singular nodes is $(1,1,2,2,3,3)$.%

\[
\begin{array}{r@{}l}Y' := &\phantom{\cup\,} \mathcal J^{(-2,0,0)}(X) \cup \{N_1^{(-2,-1,0)}\} \cup  \mathcal J^{(-1,-1,0)}(X) \cup \{N_1^{(-1,-2,0)}\}\\
&\cup\,  \mathcal J^{(0,-2,0)}(X) \cup \{N_2^{(0,-2,-1)}\} 
\cup\,  \mathcal J^{(0,-1,-1)}(X) \cup \{N_2^{(0,-1,-2)}\} \\
&\cup\,  \mathcal J^{(0,0,-2)}(X) \cup \{N_3^{(-1,0,-2)}\} 
\cup\,  \mathcal J^{(-1,0,-1)}(X) \cup \{N_3^{(-2,0,-1)}\}. 
\end{array}
\]

This example can be generalised by taking any value of $n$ and $\rho$ and an $n$-cycle $\sigma\in\s_n$ and considering the sequence 
\[(\underbrace{1,\dots,1}_{\rho\text{ times}},\underbrace{\sigma(1),\dots,\sigma(1)}_{\rho\text{ times}},\dots,\underbrace{\sigma^{n-1}(1),\dots,\sigma^{n-1}(1)}_{\rho\text{ times}}).\]

For $n\geq 4$ there exist also more complicated examples, in which any two consecutive nodes of a fine compactified Jacobian of the necklace curve correspond to sheaves that are singular 
 at two different nodes of $X$. This happens in the following example with $n=4$, $\rho=5$,
\[
\begin{array}{r@{\mspace{-5mu}}l}Y'' := &\phantom{\cup\, } \mathcal J^{(1,-1,2,1)}(X) \cup \mathcal J^{(1,0,1,1)}(X) \cup \mathcal J^{(2,-1,1,1)}(X) \cup \mathcal J^{(2,0,0,1)}(X) \cup\mathcal J^{(1,0,0,2)}(X)\\
& \cup\, \mathcal J^{(1,1,-1,2)}(X) \cup \mathcal J^{(1,1,0,1)}(X) \cup \mathcal J^{(1,2,-1,1)}(X) \cup \mathcal J^{(1,2,0,0)}(X) \cup\mathcal J^{(2,1,0,0)}(X)\\
& \cup\, \mathcal J^{(2,1,1,-1)}(X) \cup \mathcal J^{(1,1,1,0)}(X) \cup \mathcal J^{(1,1,2,-1)}(X) \cup \mathcal J^{(0,1,2,0)}(X) \cup\mathcal J^{(0,2,1,0)}(X)\\
& \cup\, \mathcal J^{(-1,2,1,1)}(X) \cup \mathcal J^{(0,1,1,1)}(X) \cup \mathcal J^{(-1,1,1,2)}(X) \cup \mathcal J^{(0,0,1,2)}(X) \cup\mathcal J^{(0,0,2,1)}(X)\\
& \cup\, \{N_1^{(0,-1,2,1)},N_2^{(1,-1,1,1)},N_1^{(1,-1,1,1)},N_2^{(2,-1,0,1)},N_4^{(1,0,0,1)},\\
& \phantom{\cup\, \{}N_2^{(1,0,-1,2)},N_3^{(1,1,-1,1)},N_2^{(1,1,-1,1)},N_3^{(1,2,-1,0)},N_1^{(1,1,0,0)},\\
& \phantom{\cup\, \{}N_3^{(2,1,0,-1)},N_4^{(1,-1,1,-1)},N_3^{(1,1,1,-1)},N_4^{(0,1,2,-1)},N_2^{(0,1,1,0)},\\
& \phantom{\cup\, \{}N_4^{(-1,2,1,0)},N_1^{(-1,1,1,1)},N_4^{(-1,1,1,1)},N_1^{(-1,0,1,2)},N_3^{(0,0,1,1)}
\},
\end{array}
\]
in which the relevant sequence is $(1,2,1,2,4,2,3,2,3,1,3,4,3,4,2,4,1,4,1,3)$.
\end{remark}

To complete the classification of Lemma~\ref{description-I_n}, we show that any smoothable \fcj of a necklace curve has the same number of components as the necklace curve itself.

\begin{proposition} \label{prop:I_n}
Every smoothable degree $d$ fine compactified Jacobian $\overline{\mathcal J}^d(X)$ of a necklace curve $X$ is isomorphic to $X$. 
\end{proposition}
For fine compactified Jacobians coming from a polarisation, this was proved in  \cite[Proposition~7.3]{meravi}. 

\begin{proof}

If $X$ is irreducible the result follows immediately from Lemma~\ref{irred1}. From now on we assume that the number of irreducible components of $X$ is $n \geq 2$.

We will use Kodaira's classification of elliptic fibrations (see e.g. \cite[Section~V.7]{BHPvdV}) to conclude. By assumption we have that $X$ is a curve of type $I_n$ and by Lemma~\ref{description-I_n}  we  know that $\overline{\mathcal J}^d(X)$  is a curve of type $I_m$ for $m=\rho n$ some integer multiple of $n$.

Let  $T=\operatorname{Spec} \C [\![t]\!]$ and consider a $1$-parameter smoothing $\mathcal{X}/T$ of $X$ over $T$ (which we always assume to have nonsingular total space) such that $\overline{\mathcal J}^d(X)$ extends to a \fcj{} $\overline{\mathcal J}^d(\mathcal{X}/T)$.  %

We claim that $\overline{\mathcal J}^d(\mathcal{X}/T)$ is nonsingular. 
Indeed, let $\mathcal F$ be a point of $\overline{\mathcal J}^d(\mathcal{X}/T)$ and let us denote by $F\in \overline{\mathcal J}^d(X)$ its central fibre.
Then by \cite[Theorem~A, (ii)]{cmkv1}, a sufficient condition for  $\overline{\mathcal J}^d(\mathcal{X}/T)$ to be nonsingular 
 at $\mathcal F$ is that the singular points of the sheaf $F$ are nodes of $X$ that are independently smoothed in $\mathcal{X}/T$.
 Hence the claim follows from the fact that, since $X$ has genus $1$,  by Proposition~\ref{simple:I_n} the sheaf $F$ is singular at most at $1$ point of $X$.

We have then two elliptic fibrations  $\mathcal{X}/T$ and $\overline{\mathcal{J}}^d(\mathcal{X}/T)$ with nonsingular total space and isomorphic generic fibres. By Kodaira's classification of elliptic fibrations, we conclude that the central fibres, respectively $X$ and   $\overline{\mathcal{J}}^d(X)$, are isomorphic, i.e. that $n=m$.
\end{proof}

We can then combine the results of  Proposition~\ref{prop:I_n} and Lemma~\ref{description-I_n} to give an explicit combinatorial description of the smoothable \fcjs of necklace curves. In the case of a \fcj, the cyclic sequence $(j_1,j_2,\dots,j_n)$ of Lemma~\ref{description-I_n} can be identified with the $n$-cycle $\sigma=(j_1\,j_2\,\cdots\,j_n)\in\s_n$ and the multidegree $\mathbf D$ is the general multidegree of the component of $Y=\overline{\mathcal J}^d(X)$ that contains both a sheaf that is singular at $P_1$ and a sheaf that is singular at $P_{\sigma^{-1}(1)}$.

\begin{corollary}\label{cor:fcjs}
For every smoothable degree $d$ \fcj $\overline{\mathcal J}^d(X)$ of a necklace curve there exist a unique multidegree $\mathbf D = (D_1,\dots,D_n)\in\Z^n$ and a unique $n$-cycle $\sigma\in\s_n$ such that 
$Y=\overline{\mathcal J}^d(X)$ is of the form~\eqref{description-Y},
where we set $j_k:=\sigma^{k-1}(1)$ for all $k\in\Z$. 
\end{corollary}

\begin{remark}
It is easy to check that a \fcj $\overline{\mathcal J}^d(X)$ contains the image of a translate of the Abel map
$X\rightarrow \Simp dX$, $p\mapsto I(p)\otimes M$ for some line bundle $M$ on $X$ of degree $d+1$, if and only if the  $n$-cycle corresponding to $\overline{\mathcal J}^d(X)$ via Corollary~\ref{cor:fcjs} is $\sigma=(1\,2\,\cdots\,n)$.
\end{remark}

We are ready to prove that all \fcjs of genus $1$ curves are defined by a polarisation.

\begin{proposition}
\label{singleispolarised}
Let $\overline{\mathcal J}^d(X)$ be a smoothable degree $d$ \fcj of  a necklace curve $X$ with $n\geq 2$ components. Then $\overline{\mathcal J}^d(X)=\overline{\mathcal J}_\phi^d(X)$ for the polarisation $\phi\in V^d(\Gamma)\cong\R^n$ given by
\begin{equation}\label{formulapol}
\phi:= \frac{1}{|S|}\sum_{\mathbf d\in S}\mathbf d 
\end{equation}
for $S = \left\{\mathbf d\in\Z^n\colon \overline{\mathcal J}^d(X) \text{ contains line bundles of multidegree } \mathbf d\right\}$.
\end{proposition}

\begin{remark}
Formula~\eqref{formulapol} is also valid when $X$ is an arbitrary nodal  curve of genus~$1$ (i.e. it is not necessarily a necklace curve), because the degree of the line bundles in $\overline{\mathcal J}^d(X)$ is constant on the rational tails. By \cite[Proposition~7.4]{meravi}, the polarisation in \eqref{formulapol} is the only polarisation $\phi$ defining $\overline{\mathcal J}^d(X)$ such that $n\phi\in\Z^n$.
\end{remark}

\begin{proof}%
If $X$ is a necklace curve with $n\geq 2$ components,
then $\overline{\mathcal J}^d(X)$ can be described as in  Corollary~\ref{cor:fcjs}. Then it is easy to check that if we define 
$\phi\in\R^n$ by 
\[\phi:=\mathbf D + \frac 1n\sum_{i=1}^n \mathbf d_i,\]
we have that $\overline{\mathcal J}^d(X)$ contains a line bundle of multidegree $(d_1,\dots,d_n)$ if and only if 
\[
0 < \left| \sum_{i=r}^s (\phi_i-d_i)\right| < 1
\]
holds for all $1\leq r\leq s\leq n$.
This implies that the condition~\eqref{Eqn: SymDefOfStability} for $\phi$-stability is satisfied for all line bundles in $\overline{\mathcal J}^d(X)$, so that by properness we have $\overline{\mathcal J}^d(X)=\overline{\mathcal J}_\phi^d(X)$. 
\end{proof}

In Table~\ref{t:4} we illustrate the correspondence between $n$-cycles, smoothable \fcjs and the polarisation $\phi$ in Proposition~\ref{singleispolarised} in the case of total degree $d=0$, number of nodes of the necklace curve $n=4$. We  assume there that the \fcj contains the generalised Jacobian --- the moduli space of line bundles of multidegree $\mathbf D = (0,0,0,0)$. This  can always be achieved up to translation by a line bundle of total degree $0$ on the necklace curve.

\begin{table}
\caption{Smoothable fine compactified Jacobians of a curve of type~$I_4$ \label{t:4}}
\scalebox{0.8}{$
\begin{array}{cccccc}
\text{permutation}&\mathbf d_1&\mathbf d_2&\mathbf d_3&\mathbf d_4&\phi\\\hline
\\[-0.5em]
(1\;2\;3\;4)&(1,-1,0,0)&(1,0,-1,0)&(1,0,0,-1)&(0,0,0,0)&\left(\frac 34,-\frac 14,-\frac 14,-\frac14\right)
\\[0.5em]
(1\;2\;4\;3)&(1,-1,0,0)&(1,0,-1,0)&(0,0,-1,1)&(0,0,0,0)&\left(\frac 12,-\frac 14,-\frac 14,\frac14\right)
\\[0.5em]
(1\;3\;2\;4)&(1,-1,0,0)&(1,-1,1,-1)&(1,0,0,-1)&(0,0,0,0)&\left(\frac 34,-\frac 12,\frac 14,-\frac12\right)
\\[0.5em]
(1\;3\;4\;2)&(1,-1,0,0)&(1,-1,1,-1)&(0,-1,1,0)&(0,0,0,0)&\left(\frac 12,-\frac 34,-\frac 12,-\frac14\right)
\\[0.5em]
(1\;4\;2\;3)&(1,-1,0,0)&(0,-1,0,1)&(0,0,-1,1)&(0,0,0,0)&\left(\frac 14,-\frac 12,-\frac 14,\frac12\right)
\\[0.5em]
(1\;4\;3\;2)&(1,-1,0,0)&(0,-1,0,1)&(0,-1,1,0)&(0,0,0,0)&\left(\frac 14,-\frac 34,\frac 14,\frac14\right)
\end{array}
$}
\end{table}

\smallskip

The combinatorial characterisation in Proposition~\ref{f-for-fcjs} comes from the study of the stability cell containing $\phi$. 

\begin{proof}[Proof of Proposition~\ref{f-for-fcjs}]
We start by showing that every \fcj{} $\overline{\mathcal J}^d(X)$ on a necklace curve $X$ with $n$ components arises from a function $f\colon\mathcal C_{n-1} \rightarrow \Z$ that satisfies Condition~\ref{condfunction}. Let us recall that we have $\overline{\mathcal J}^d(X)=\overline{\mathcal J}^d_\phi(X)$ for the stability condition $\phi$ given by the average of the multidegrees of line bundles in $\overline{\mathcal J}^d(X)$, as described in Proposition~\ref{singleispolarised}.
Consider the stability cell containing $\phi$.
After identifying $V^d(\Gamma(X))$ with $\R^{n-1}$ by forgetting the last component of $\phi$, 
the stability cell containing $(\phi_1,\dots,\phi_{n-1})$ is defined by the inequalities

\begin{equation}\label{eq-fandq}
c_{r,s} < \sum_{i=r}^s q_i < c_{r,s}+1
    \end{equation}
where the integers $c_{r,s}$ 
correspond to the values $f_{\overline{\mathcal J}^d(X)}\left(\{r,r+1,\dots,s\}\right)$ of the following function:
\[
\begin{array}{rrcl}
f_{\overline{\mathcal J}^d(X)}\colon&\mathcal C_{n-1}&\longrightarrow&\Z\\
&I&\longmapsto&\min\left\{\sum_{i\in I}d_i:\; \mathcal J^{(d_1,\dots,d_n)}(X)\subset \overline{\mathcal J}^d(X)\right\}.
\end{array}
\]
 By using the explicit characterisation in Corollary~\ref{cor:fcjs}, one can check that $f_{\overline{\mathcal J}^d(X)}$ satisfies Condition~\eqref{condfunction} (mild superadditivity). 

Furthermore, by induction on $n$ it is possible to show  that 
there exist exactly $(n-1)!$
functions $f\colon\mathcal C_{n-1}\rightarrow \Z$ that satisfy \eqref{condfunction} with prescribed values $f(\{1\}),\dots,f(\{n-1\})$.  Since by Corollary~\ref{cor:fcjs} there exist exactly $(n-1)!$ smoothable \fcjs{} of degree $d$ up to translation, this yields that every~$f$ must give rise to a \fcj.
\end{proof}

\section{A stratification of genus \texorpdfstring{$1$}{1} fine compactified universal Jacobians}\label{s:strataofJmb}

In this section we fix $n\geq 1$ and $d \in \mathbb{Z}$, and study the topology of each genus~$1$  fine compactified universal Jacobian $\Jmb d1n$. The main result is Lemma~\ref{stratawise}, where we stratify each such \fcuj into strata that are isomorphic to strata of $\Cmb 1n$, the universal curve over $\Mmb 1n$.

Let us recall from Corollary~\ref{forgettails} that 
the Jacobian of a stable curve of genus~$1$ is the same if we remove its rational tails. 
This leads us to introduce the following substack of $\Mmb 1n$.
\begin{definition} 
  We denote by $\NR n$ the constructible substack of $\Mmb 1n$ consisting of necklace curves (defined in~\ref{necklace}).
  \end{definition}
  (The superscript in $\NR n$ is motivated by the fact that among the stable curves of genus~$1$, necklace curves are precisely those that do not have any rational tails).

Any stable  $n$-pointed  curve $X$ of genus~$1$ can be obtained by attaching rational tails 
 to the marked points of a necklace curve $X'$. This corresponds to the fact that $\Mmb 1n$ can be stratified as the disjoint union of constructible subsets
\begin{equation}\label{stratum-mb1n}
\NR k \times \prod_{j=1}^k\Mmb 0{I_j\cup\{*\}} 
\end{equation}
for any partition of $\{1,\dots,n\}$ into nonempty subsets $I_1,\dots,I_k$. We can avoid repetitions in the stratification \eqref{stratum-mb1n} by stipulating an ordering convention %
on the subsets $I_1,\dots,I_k$, for example $\max I_1 < \max I_2<\cdots <\max I_k$.
We also adopt the convention that $\Mmb 0{\{m,*\}}$ is a point, so that the case where one of the $I_j$ is the singleton $\{m\}$ corresponds to the case where the $m$th marked point lies on the necklace curve $X'$.
Let us recall that for each choice of $(k;I_1,\dots,I_k)$, the corresponding constructible subset \eqref{stratum-mb1n} is a union of strata of the stratification by topological type.

The main result of this section is a decomposition of each genus~$1$ \fcuj into easier pieces:
\begin{lemma} \label{stratawise}
  Let $\Jmb d 1n\rightarrow \Mmb 1n$ be a \fcuj and let $\mathcal S$ be a stratum of the stratification by topological type of $\Mmb 1n$ consisting of %
a necklace curve with $k$ maximal rational rails attached to it,
  i.e. $\mathcal S \subset \NR k \times \prod_{j=1}^k\Mmb 0{I_j\cup\{*\}}$.
  Then the restriction $\Jmb d 1n|_\mathcal S$ of $\Jmb d 1n$ to $\mathcal S$ is isomorphic to the pull-back of the universal family $\CNR k\rightarrow \NR k$ under the forgetful map $\mathcal S\rightarrow \NR k$.
\end{lemma}

Our proof will depend upon the number of the number of components of the necklace curve of each stratum. When the number of components is at most $2$, our isomorphism will be given by an Abel map. In the other cases the isomorphism is obtained from Proposition~\ref{prop:I_n} and it is for that reason noncanonical, even after ordering  the components of the necklace curve. %
\begin{proof}%
Let $\mathcal S \subset \NR k \times \prod_{j=1}^k\Mmb 0{I_j\cup\{*\}}$ be a stratum of $\Mmb 1n$ corresponding to a fixed topological type.
 If we denote by $\mathcal N\subset \NR k$ the stratum corresponding to the topological type of the necklace subcurve and by $\mathcal T_j\subset \Mmb 0{I_j\cup\{*\}}$ for $j=1,\dots,k$ the strata corresponding to the topological type of the rational tails, we have 
\[\mathcal S \cong \mathcal N \times \prod_{j=1}^k\mathcal T_j.\]

Since the fine compactified Jacobian of a rational tail is a point,  the fibre of $\Jmb d1n|_\mathcal S \rightarrow \mathcal S$ at some $[X,p_1,\dots,p_n]$ is isomorphic to the \fcj, for some degree $d'$, of the unique necklace subcurve $X'\subseteq X$. By Lemma~\ref{separatingnode} this description extends to the whole family $\Jmb d1n|_\mathcal S$, so that there is an isomorphism 
\[\Jmb d1n|_\mathcal S \cong \overline{\mathcal{J}}^{d'}(U_\mathcal N/\mathcal N)\times \prod_j\mathcal T_j\]
for some $d' \in \mathbb{Z}$ and some \fcj{} $\overline {\mathcal{J}} ^{d'}(U_\mathcal N/\mathcal N)$ for the restriction $U_\mathcal N\rightarrow \mathcal N$ of the universal family $\Cmb 1k / \Mmb 1k$ to the stratum $\mathcal N$. To conclude we need to exhibit an isomorphism $U_\mathcal N/ \mathcal N \to \overline {\mathcal{J}} ^{d'}(U_\mathcal N/\mathcal N)$.

From now on we let $r$ be the number of irreducible components (necessarily rational) of $X'$. To simplify the notation, we will relabel the marked points on $X'$ and call them $q_1, \ldots, q_k$, and we will assume that the first $r$ of them lie on different components of $X'$. 

If the necklace subcurve $X'$ is irreducible, then an isomorphism $U_\mathcal N/ \mathcal N \to \overline {\mathcal{J}} ^{d'}(U_\mathcal N/\mathcal N)$ is given by an Abel map as shown in Lemma~\ref{irred1}.

The case $r=2$ is settled in a similar manner. A direct analysis shows that for each fine compactified Jacobian $\overline{\mathcal{J}}^{d'}(X')$ there exist integers $d_1,d_2$ with $d_1+d_2=d'+1$, such that the Abel map \[p \mapsto \mathcal{O}_{X'}(d_1 q_1 + d_2 q_2)\otimes I(p)\] induces an isomorphism $X' \to \overline{\mathcal{J}}^{d'}(X')$.

We still have to deal with the case when $X'$ consists of $r \geq 3$ irreducible components. In this case the pointed curve $(X', q_1, \ldots, q_r)$ has no nontrivial automorphisms, so each fibre of $U_{\mathcal{N}} \to \mathcal{N}$ admits a unique isomorphism to $(X', q_1, \ldots, q_r)$ that respects the sections,  and this exhibits $U_{\mathcal{N}} \to \mathcal{N}$ as the trivial family.  By Lemma~\ref{product} the family $\overline{\mathcal{J}} ^{d'}(U_\mathcal N/\mathcal N)$ is also trivial and given by the product of $\mathcal N$ and some smoothable \fcj $\overline{\mathcal{J}}^{d'}(X')$ of $X'$. 
The proof is then concluded by observing that, by Proposition~\ref{prop:I_n}, the fine compactified Jacobian $\overline{\mathcal{J}}^{d'}(X')$ is isomorphic to $X'$.\end{proof}

Lemma~\ref{stratawise} allows to introduce a  noncanonical refinement of the stratification of each \fcuj $\Jmb d1n$ by topological type of the underlying moduli space of stable pointed curves. %

\begin{corollary} \label{strata}
Each \fcuj $\Jmb d1n$  can be  stratified, using the isomorphisms of Lemma~\ref{stratawise}, into a refinement of the inverse image of the topological type strata of $\Mmb 1n$ under the forgetful map. Each stratum of $\Jmb d1n$ corresponds to a stratum of $\Cmb 1n\cong \Mmb 1{n+1}$ given by curves $(X,p_1,
\dots,p_{n+1})$ such that the stabilisation of $(X,p_1,\dots,p_{n})$ has a fixed topological type, and the point $p_{n+1}$  lies on either 
\begin{itemize}
    \item[(a)]  an irreducible component of $X$ that is not contained in any rational tail, or
    \item[(b)]  
    an irreducible and maximal rational tail of $X$ containing only one other marked point $p_j$, or
\item[(c)] one of the maximal rational tails, on the  unique component of that rational tail that intersects the necklace subcurve.
\end{itemize} 
\end{corollary}

\begin{proof}
This is the stratification by topological type of $\Mmb 1{n+1}$, induced on $\Jmb d1n$ via the isomorphisms of Lemma~\ref{stratawise} on each restriction  $\Jmb d1n|_{\mathcal{S}}$, for  $\mathcal{S}$ a topological type stratum of $\Mmb 1n$. %
\end{proof}

To conclude, we observe that the above stratification is a noncanonical refinement of a canonical one.

\begin{remark} If $X$ is a necklace curve, then we have proved in \ref{prop:I_n} that every smoothable fine compactified Jacobian $\overline{\mathcal{J}}^d(X)$ is isomorphic to $X$. The curve $X$ can be naturally stratified by its irreducible components and singular locus. Whilst the isomorphism $X \cong \overline{\mathcal{J}}^d(X)$ is in general noncanonical, the induced stratification of $\overline{\mathcal{J}}^d(X)$ is canonical. 

If $[X, p_1, \ldots p_n] \in \Mmb 1n$, then the stratification given in Corollary~\ref{strata} induces a stratification of $\overline{\mathcal{J}}^d(X)$ that noncanonically refines the one described in the previous paragraph. The stratification from Corollary~\ref{strata} contains $n$ additional strata of type (b). These additional strata depend upon the  choice of an isomorphism $X \cong \overline{\mathcal{J}}^d(X)$, and that choice is noncanonical whenever $n \geq 3$. 
\end{remark}

\section{Cohomology of \texorpdfstring{$\Jmb d1n$}{fine compactified universal Jacobians}}

In this section we fix $n \geq 1$ and $d \in \mathbb{Z}$ and use the results of Section~\ref{s:strataofJmb} to calculate the rational cohomology of every genus~$1$ fine compactified universal Jacobian $\Jmb{d}{1}{n}$. As expected, the result only depends on $n$ --- it is independent of $d$ and of the particular fine compactified Jacobian. This is analogous to the main result of \cite{mishvi}, which states that the cohomology of
\emph{polarised} fine compactified Jacobians of a single curve of any genus is independent of the polarisation.

Our main tool will be Lemma~\ref{stratawise}. An important role will be played by the fact that every \fcuj is \emph{nonsingular} as a Deligne--Mumford stack, a fact that follows from the fact that it is open in the nonsingular moduli stack $\operatorname{Simp}^{d}(\Cmb{g}{n} / \Mmb gn)$.

We start by explaining how the even cohomology of genus~$1$ fine compactified universal Jacobians admits a geometric interpretation with the strata of Corollary~\ref{strata}.  

\begin{corollary}\label{boundary}
The classes of the cycles of Corollary~\ref{strata} span  the even cohomology of $\Jmb d{1}{n}$. 
In particular, the even cohomology of $\Jmb d{1}{n}$ is all algebraic.
\end{corollary}

\begin{proof}
The proof is completely analogous to the proof that the even cohomology of $\Mmb 1n$ is generated by cycle classes of strata, given by Petersen in \cite[Section 1]{petersen}. In both cases, the claim follows from the fact that the total space --- in Petersen's case $\Mmb 1n$, in our case $\Jmb d1n$ --- is complete and nonsingular (as a Deligne--Mumford stack), so that its rational cohomology satisfies Poincar\'e duality, combined with a Hodge-theoretic analysis of the cohomology of the strata, ensuring that the only cohomology classes in the even cohomology of each stratum with Hodge weight equal to the degree are in degree $0$. Since the strata of $\Jmb d1n$ are isomorphic to certain strata of $\Mmb1{n+1}$, this Hodge-theoretic analysis applies to the strata of $\Jmb d1n$ as well.
\end{proof}

We now aim at calculating the cohomology of any \fcuj $\Jmb d 1n$, using Lemma~\ref{stratawise}. Let us recall that, by definition, the complex Deligne--Mumford stack $\Jmb d 1n$ is smooth and proper. In particular, it satisfies Poincar\'e duality and as a consequence, we have that the Hodge structures on $\coh k {\Jmb d 1n}$ are pure of Hodge weight equal to $k$. This means that we can recover the structure of graded $\Q$-vector space and the Hodge structures on $\coh\pu{\Jmb d 1n}$  from its Euler characteristic in the Grothendieck group $K_0(\HS_\Q)$ of rational Hodge structures. Let us recall that the $\s_n$-action on $\Jmb d 1n \rightarrow \Mmb 1n$ by permuting the $n$ marked points endows the cohomology of $\Jmb d1n$ with a structure of $\s_n$-representation. To keep track of this as well, we introduce the following definition.

\begin{definition}
  Let $X$ be a complex quasi-projective variety. We define the $\s_n$-equivariant Hodge Euler characteristic $\hodgechar nX$ of $X$ as the alternating sum of the class of the $i$th cohomology of compact support of $X$ in the Grothendieck group of rational Hodge structures with a compatible structure as $\s_n$-representation, that is
  \[
\hodgechar nX = \sum_{i\in\Z}(-1)^i [\cohc i X] \in K_0^{\s_n}(\HS_\Q)\cong \Lambda_n \otimes_\Q K_0(\HS_\Q),
\]
where $\Lambda_n$ is the space of symmetric functions of degree $n$, with rational coefficients.

If $X$ is a constructible subset with an $\s_n$-action, we can define its $\s_n$-equivariant Hodge Euler characteristic additively as
\[\hodgechar nX = \sum_{j\in J}\hodgechar n{X_j},\]
where $\{X_j\}_{j\in J}$ is any stratification of $X$ into $\s_n$-invariant quasi-projective varieties $X_j$.

If $\mathcal X$ is a constructible subset of a Deligne--Mumford stack with a quasi-projective coarse moduli space, then if there is a $\s_n$-action on $\mathcal X$ we can define the $\s_n$-equivariant Hodge Euler characteristic of $\mathcal X$ to be $\hodgechar nX$ where $X$ is the coarse moduli space of $\mathcal X$.

\end{definition}

Hodge Euler characteristics are often used to describe the cohomology of $\Mmb gn$. One of their advantages is that they allow to use operations on the ring
$\Lambda = \bigoplus_{n\geq 0}\Lambda_n$ of symmetric functions. For instance, this is the foundation of Getzler's formula for the cohomology of $\Mmb 1n$. Before we proceed, let us set up some notation for the generating functions of the $\s_n$-equivariant Hodge Euler characteristics of moduli spaces of curves in genus $g=0,1$ and for the locus of necklace curves in genus~$1$:
\begin{align*}
  \bfa_g &:= \sum_{n\geq 3-2g}\hodgechar n{\Mm gn},&
  \bfb_g &:= \sum_{n\geq 3-2g}\hodgechar n{\Mmb gn},\\
  \bfb_1^{\operatorname{NR}}&:= \sum_{n\geq 1}\hodgechar n{\NR n}.
\end{align*}

Let us recall Getzler's result on the cohomology of $\Mmb 1n$:
\begin{theorem}[{\cite[Theorem~(2.6)]{getzler-semiclassical}}, \cite{dan-semiclassical}] \label{thm-getzler}
In genus~$1$, the generating functions for the Hodge Euler characteristic of the moduli space of necklace curves and stable curves are given respectively by:
  \begin{align}
 \bfb_1^{\operatorname{NR}}&=
\bfa_1 - \frac 12\sum_{n\geq 1}\frac{\phi(n)}n \log(1-\psi_n(\bfa_0''))+\frac{\dot\bfa_0(1+\dot\bfa_0)+\frac14\psi_2(\bfa_0'')}{1-\psi_2(\bfa_0'')}, \label{getzler-1}
\\
    \bfb_1 &= \bfb_1^{\operatorname{NR}}\circ (p_1+\bfb_0'), \label{getzler-2}
    \end{align}
  where for a symmetric function $f$, we write $f'=\frac{\partial f}{\partial p_1}$ and
  $\dot f=\frac{\partial f}{\partial p_2}$ for its derivatives with respect to the power sums $p_1$ and $p_2$, respectively,
  and for all $k\geq 0$ we define the $k$th Adams operation by $\psi_k(f) := p_k \circ f$. %
\end{theorem}
\begin{theorem} \label{main}
  Let $\Jmb d 1n$ be a \fcuj over $\Mmb 1n$ for all $n\geq 1$. Then the generating function for the $\s_n$-equivariant Hodge Euler characteristic of $\Jmb d 1n$ is given by
  \begin{equation}\label{mainformula}
\sum_{n\geq 1}\hodgechar n{\Jmb d 1n} =
\left((1+p_1)(\bfb_1^{\operatorname{NR}})'\right)\circ(p_1+\bfb_0'),\end{equation}
where $\bfb_1^{\operatorname{NR}}$ is given in \eqref{getzler-1}.
\end{theorem}

In the appendix we include some tables of the $\s_n$-equivariant Betti numbers of \fcujs{} in genus $1$.

Let us recall from  \eqref{stratum-mb1n} that $\Mmb 1n$ can be stratified according to the number $k$ of rational tails of the curves into constructible subsets of the form
\[\NR k \times \prod_{j=1}^k\Mmb 0{I_j\cup\{*\}}, \]
where we choose  to interpret marked points on the necklace subcurve as rational tails with a single marked point
by setting $\Mmb 0{\{m,*\}}:=\Mmb 02=\{\text{point}\}$.

\begin{proof}
Our statement follows from Lemma~\ref{stratawise} combined with  Theorem~\ref{thm-getzler}. We therefore start by recalling Formula \eqref{getzler-2}. Let us recall that the expression $p_1+\bfb_0'$ is the generating function
\[
\sum_{n\geq 1}\hodgechar n{\Mmb 0{\{1,\dots,n\}\cup\{*\}}}\]
where $\Mmb 0{\{1,\dots,n\}\cup\{*\}}$ is considered with the natural action of the symmetric group $\s_n$. Since the $\s_n$-representation on $\coh\pu{\Mmb 0{\{1,\dots,n\}\cup\{*\}}}$ is given by restriction of the $\s_{n+1}$-representation on $\coh\pu{\Mmb 0{n+1}}$, we have
  \[
  [\coh k{\Mmb 0{\{1,\dots,n\}\cup\{*\}}}] = \frac \partial{\partial p_1}[\coh k{\Mmb 0{n+1}}],
  \]
  where by convention we have $[\Mmb 0{\{1,*\}}]= p_1$.

  Thus, taking the plethysm of $\bfb_1^{\operatorname{NR}}$ with $p_1+\bfb_0'$ gives as a result the generating series of the $\s_n$-equivariant Hodge Euler characteristic of the space parametrising necklace curves with any possible number of rational tails attached. For instance, by definition of plethysm the summand $\hodgechar k {\NR k} \circ (p_1+\bfb_0')$ gives the generating series of the equivariant Hodge Euler characteristics of the strata consisting of curves of genus~$1$ with exactly $k$ maximal rational tails attached:
  \[
\hodgechar k {\NR k} \circ (p_1+\bfb_0') = \sum_{\substack{n\geq k\\ I_1\sqcup I_2 \sqcup \dots \sqcup I_k = \{1,\dots,n\}}}\hodgechar n{\NR k \times \prod_{j=1}^k \Mmb 0{I_j\cup\{*\}}}.
  \]

Now that we understand why  Getzler's Formula~\eqref{getzler-2} holds,
all we need to do is to adapt the idea behind Formula~\eqref{getzler-1} to keep track of all possible compactified Jacobians in $\Jmb d1n$. Since Hodge Euler characteristics are additive, it is enough to work stratawise under the stratification \eqref{stratum-mb1n} of $\Mmb 1n$. Hence, let us fix a stratum $\mathcal S = \NR k \times \prod_{j=1}^k\Mmb 0{I_j\cup\{*\}}$ of $\Mmb 1n$. Then, by Lemma~\ref{stratawise}, the preimage $\pi^{-1}(\mathcal S)$ of $\mathcal S$ under $\Jmb d1n\rightarrow \Mmb 1n$ is isomorphic the moduli stack parametrizing $(k+1)$-tuples $(\mathcal C,\mathcal T_1,\dots,\mathcal T_k)$ where \begin{itemize}
\item[-] $\mathcal C=(C,t_1,\dots,t_k,p)$ is a fibre of the universal family over $\NR k$, i.e. a necklace curve $C$ of genus~$1$ with $k$ distinct ordered marked points $t_1,\dots,t_k$ and an additional point $p$ which can lie anywhere on $C$;
\item[-] for all $j=1,\dots,k$, the stable rational curve $\mathcal T_{j} = (T_j,(p_i)_{i\in I_j},p_*)\in\Mmb 0{I_j\cup\{*\}}$ represents the rational tail attached to $C$ at the point $t_j\in C$ by identifying $p_*\in T_j$ with $t_j\in C$.
\end{itemize}

Hence, to obtain a formula for the generating function of the Hodge Euler characteristics of $\Jmb d1n$, we need to replace $\bfb_1^{\operatorname{NR}}$ in Formula ~\eqref{getzler-2} with the generating function
\[\bfc_1^{\operatorname{NR}} :=\hodgechar k{\CNR k}\]
of the universal family $\CNR k \rightarrow \NR k$. Here, we have to distinguish between two cases. The first case is the general one and defines an open substack $U\subset \CNR k$. It corresponds to the case in which the point $p$ is distinct from the marked points $t_1,\dots,t_k$ at which the elliptic tails are attached. Then $U$ can be identified with the point $(C,t_1,\dots,t_k,p)\in\NR{k+1}$ and its contribution to $\bfc_1^{\operatorname{NR}}$ is exactly
\[\frac{\partial \bfb_1^{\operatorname{NR}}}{\partial p_1} = (\bfb_1^{\operatorname{NR}})'.\]

The other case is the case in which $p$ coincides with one of the other marked points. This gives rise to a closed substack of $\CNR k$ which is the disjoint union of $k$ copies of $\NR k$ indexed by $j=1,\dots,k$. To obtain the contribution of $\CNR k\setminus U$ to $\bfc_1^{\operatorname{NR}}$, we need first to consider $k$ copies of $\NR k$ with the action of $\s_{k-1}$ and then extend the corresponding representations from $\s_{k-1}$ to $\s_k$. In terms of Schur polynomials, this means that we get the additional contribution
\[p_1 \frac{\partial \bfb_1^{\operatorname{NR}}}{\partial p_1} = p_1(\bfb_1^{\operatorname{NR}})'.\]

Combining all information obtained so far yields the claim.
\end{proof}

For the convenience of the reader, we exhibit  in Tables~\ref{table-jmb-1} and~\ref{table-jmb} at the end of this paper the $\s_n$-equivariant Hodge Euler characteristic of $\Jmb d1n$ for all $n\leq 8$.
Observe that the Hodge Euler characteristics have the same coefficients for Hodge weight (and degree) $k$ and $2n-k$, in accordance with Poincar\'e duality. One can compare them with the cohomology of $\Mmb 1{n+1}$ which is in Table~\ref{table-mmb}. All tables were obtained by implementing in Sage \cite{sagemath} the formulas in Theorem~\ref{main} and \ref{thm-getzler}.

\section{Classification of genus \texorpdfstring{$1$}{1} fine compactified universal Jacobians}
\label{fcuj-classification}
In this section we give an explicit combinatorial classification of all genus~$1$ fine compactified universal Jacobians of some fixed degree $d\in \mathbb{Z}$ building on Proposition~\ref{f-for-fcjs}, the analogous statement  for necklace curves. From this description we  deduce that not all fine compactified universal Jacobians can be constructed from some universal polarisation as  in \cite{kp2} or in \cite{melouniversal}, see Definition~\ref{phistab}. %

Our classification is based on associating with each degree~$d$ \fcuj{} a pair $(f,g)$ of integer-valued functions. Along the way we will prove in Lemma~\ref{g-for-rattails} that the function $g$ describes the restriction of the multidegree of each element of the fine compactified Jacobian to the rational tails of the curve, and that the function $f$ encodes the information about the multidegrees of line bundles on all necklace curves.

Both $f$ and $g$ are defined by looking at all possible pairs $(X,L)\in\Jmb d1n$ where $L$ is a line bundle and $X$ is a curve with two components. 
In the case of $f$, one considers all stable necklace curves $X$ with two rational components and the values of $f$ are given by the minimal degree of the restriction of a line bundle $L$ to the component not containing the last marked point $p_n$. In the case of $g$, one considers all  curves $X$ with two components separated by one node; the values of $g$ are then obtained as the degree of the restriction of $L$ to the rational component.

\begin{notation}
\begin{enumerate}
    \item For every $n\geq 1$, we denote the set  of all nonempty subsets of $\{1, \ldots, n-1\}$ by $\mathcal{P}^+_{n-1}$ and the set of all subsets of $\{1,\dots,n\}$ containing at least $2$ elements by $\mathcal Q^+_n$.
\item
    We say that a function $f\colon \mathcal  P^+_{n-1}\rightarrow \Z$
is \emph{mildly superadditive}  if it satisfies
\begin{equation}\label{superadd}
0\leq f(I\cup J) - f(I) - f(J) \leq 1 \text{ for all }I,J\subset\{1,\ldots, n-1\}, I \cap J = \emptyset.
\end{equation}
\end{enumerate}  
\end{notation}

If $\Jmb d1n\subset\Simp d{\Cmb 1n/\Mmb 1n}$ is a fixed \fcuj{}, we  denote by $\overline{\mathcal J}^d(X)$ its fibre over the point $[X]\in\Mmb 1n$ representing the stable $n$-pointed curve $X$. We can use the multidegrees of the line bundles in $\overline{\mathcal J}^d(X)$ to define a pair of functions $(f,g)$ as we described above:

\begin{definition}\label{def:fandg}
For every $T\in\mathcal Q^+_n$, let us fix a curve $(X(T),p_1,\dots,p_n)\in\Mmb 1n$ with two components $Y(T)$ and $Z(T)$ of genus $1$ and $0$, respectively, where the rational component $Z(T)$ contains exactly the marked points in $T$. Analogously, for every $I\in\mathcal P^+_{n-1}$ we choose an $n$-pointed  stable curve $X_I=Y_I \cup Y'_I$ consisting of two rational components meeting at two nodes, where the component $Y_I$ contains exactly the marked points in $I$.

If $\Jmb d1n$ is a \fcuj, we define the associated function $g_{\Jmb d1n}\colon \mathcal Q^+_n\rightarrow \Z$ by setting $g_{\Jmb d1n}(T)$ to be $\deg_{Z(T)}L$ where $L$ is any line bundle in $\overline{\mathcal J}^d(X(T))$.
We then  set %
$f_{\Jmb d1n}\colon \mathcal P^+_{n-1}\rightarrow \Z$ to be the function defined by 
\begin{equation}\label{definef} f_{\Jmb d1n}(I) = \min\{\deg_{Y_I}L\co L \text{ line bundle in }\overline{\mathcal J}^d(X_I)\}.
\end{equation}
\end{definition}

We can now state the main result of this section.
\begin{theorem}\label{f-and-g}
For every $d \in \mathbb{Z}$, the map that associates with every $\Jmb d1n\subset\Simp d{\Cmb 1n/\Mmb 1n}$ the pair of integer-valued functions $(f_{\Jmb d1n}, g_{\Jmb d1n})$, defined in Definition~\ref{def:fandg}, is a bijection between the set of all degree~$d$ \fcujs{}  over $\Mmb{1}{n}$, and the set of pairs of functions $f\colon \mathcal{P}^+_{n-1}\to \mathbb{Z},g\colon \mathcal{Q}^+_{n}\to \mathbb{Z}$, such that $f$ is mildly superadditive (as defined in \eqref{superadd}). 
\end{theorem}
This explicit combinatorial description allows us to construct in Example~\ref{exotic}, for every $n \geq 6$, a genus~$1$ fine compactified universal Jacobian that is not equal to $\Jmb d1n (\phi)$ for any choice of a nondegenerate $\phi \in V_{1,n}^d$. (See Section~\ref{notallpolar} for the notion of $\phi$-polarised fine compactified Jacobians).

As a first step towards the proof of Theorem~\ref{f-and-g}, we consider the discrete data associated with the fibres $\overline{\mathcal J}^d(X)$ of $\Jmb d1n$
 over any stable curve $X$, where $\Jmb d1n$ is a fixed \fcuj{}.
It follows from Corollary~\ref{forgettails} and Corollary~\ref{cor:fcjs} that $\overline{\mathcal J}^d(X)$
is a necklace curve with $r$ components (where $r$ is the number of components of the curve $X'$ obtained from $X$ by contracting all its rational tails) and that each component $J_i$ of $\overline{\mathcal{J}}^d(X)$ is identified by the multidegree $\mathbf d:= \deg L\in \Z^{\Vrt(\Gamma(X)}$  of some (equivalently all) line bundle(s) $L\in J_i$.

Since $\Jmb d1n$ is flat over $\Mmb 1n$,  if $X$ and $X'$ are stable curves with the same topological type, then the collection of multidegrees of line bundles in $\overline{\mathcal J}^d(X)$ is the same as the collection of multidegrees of line bundles in $\overline{\mathcal J}^d(X')$. This motivates the following definition:

\begin{definition}\label{sgamma}
Let $\Jmb d1n$ be a \fcuj{}. Then for each stable graph $\Gamma \in G_{1,n}$, denote by $S_\Gamma(\Jmb d1n) = \{\mathbf d_1,\dots,\mathbf d_{r(\Gamma)}\}\subset\Z^{\Vrt(\Gamma)}$ the collection of multidegrees of line bundles in $\Jmb d1n$ on curves of type $\Gamma$, where $r(\Gamma)$ is the number of vertices of $\Gamma$ that are not contained in any rational tail.
\end{definition}

The above definition assigns to every degree~$d$ \fcuj $\Jmb d1n$ the datum  
\begin{equation}\left\{ (\Gamma,S_\Gamma(\Jmb d1n))\co \Gamma \in G_{1,n}, S_\Gamma(\Jmb d1n)\subset \mathbb{Z}^{\operatorname{Vert}(\Gamma)}\right\}\label{correspondence}\end{equation}
of subsets of $\mathbb{Z}^{\operatorname{Vert}(\Gamma)}$ for every stable $n$-pointed graph $\Gamma$ of genus $1$.

We will now show how every degree~$d$ \fcuj can be reconstructed from an assignment $\left\{(\Gamma, S_\Gamma)\co\Gamma \in G_{1,n},S_\Gamma\subset \mathbb{Z}^{\operatorname{Vert}(\Gamma)}\right\}$ of line bundle multidegrees 
that satisfies two constraints that we are now going to describe.

The first constraint is that, for every $\Gamma\in G_{1,n}$, the set $S_{\Gamma}\subset \mathbb{Z}^{\operatorname{Vert}(\Gamma)}$ should satisfy the conditions described in Corollaries~\ref{forgettails} and \ref{cor:fcjs} to be the collection of line bundle multidegrees of some degree~$d$ fine compactified Jacobian of some curve (equivalently all curves) of topological type $\Gamma$.

The second constraint is imposed by the fact that $\Jmb d1n$ parametrises \emph{flat} families of sheaves. For this reason, the assignment $(\Gamma,S_\Gamma)$
of line bundle multidegrees on curves with dual graph $\Gamma$ 
should be compatible with curves degeneration, 
i.e. with morphisms between dual graphs.
Specifically, if $\gamma\colon \Gamma\rightarrow \Gamma'$ is a morphism in $G_{1,n}$, then the datum of \eqref{correspondence} satisfies
\begin{equation}\label{functorial} S_{\Gamma'} = \left\{d'(w) := \sum_{v\in\gamma^{-1}(w)} d(v)\colon\mathbf d \in S_\Gamma\right\}.\end{equation}

The two constraints on line bundle multidegrees are summarized in the following definition.
\begin{definition} \label{compatible} 
Let $\left\{(\Gamma,S_\Gamma)\co\Gamma \in G_{1,n}, S_\Gamma\subset \mathbb{Z}^{\operatorname{Vert}(\Gamma)} \right\}$
be an assignment, for each topological type $\Gamma$, of the degree~$d$ line bundle multidegrees of a \fcj of a stable curve of topological type $\Gamma$. We say that the assignment is  \emph{compatible} if \eqref{functorial} holds for all morphisms $\Gamma\to \Gamma'$ in $G_{1,n}$. The same condition can be required of any full subcategory of $G_{1,n}$ such as the subcategory $G_{1,n}^{\operatorname{NR}}$ of genus~$1$ of stable graphs without rational tails.
\end{definition}
 The notion of compatibility is analogous to the compatibility that was required in Definition~\ref{spaceofpolar} for elements of $V_{g,n}^d$.

We are now ready to discuss the inverse of the correspondence described in ~\eqref{correspondence}.

\begin{lemma}\label{existenceofj}
Let $\left\{(\Gamma,S_\Gamma)\co\Gamma \in G_{1,n}, S_\Gamma\subset \mathbb{Z}^{\operatorname{Vert}(\Gamma)} \right\}$ 
be a compatible assignment of line bundle multidegrees of degree~$d$ fine compactified Jacobians on each topological type of stable $n$-pointed  curves of genus $1$.
Then there exists a unique \fcuj{} $\Jmb d1n\subset \Simp d{\Cmb 1n/\Mmb 1n}$ such that $S_\Gamma = S_\Gamma(\Jmb d1n)$ holds.
\end{lemma}

\begin{proof}
Uniqueness follows immediately as a consequence of the description in Corollaries~\ref{forgettails} and \ref{cor:fcjs}. Indeed, the collection of line bundle multidegrees of a fine compactified Jacobian $\overline{\mathcal J}^d(X)$ of a stable curve $X$ of genus $1$ %
contains enough information that from it we can recover the multidegrees of all points of $\overline{\mathcal J}^d(X)$. %

To show existence we will prove that the moduli space $\mathcal{Y} \subset \Simp d{\Cmb1n/\Mmb 1n}$ of sheaves corresponding to the assignment $\left\{(\Gamma,S_\Gamma)\co\Gamma \in G_{1,n}, S_\Gamma\subset \mathbb{Z}^{\operatorname{Vert}(\Gamma)} \right\}$  is open and proper. 
We will achieve this by showing how $\mathcal Y$ can be defined by gluing together the restriction  of \fcujs of the form $\overline{\mathcal J}^d_{1,n}(\phi)$ to certain open subsets of $\Mmb 1n$. 

Let $T_{1,n} \subset G_{1,n}$ be the subset of graphs with $n$ edges.
 Then the stratum of $\Mmb 1n$ corresponding to each $\Gamma\in T_{1,n}$ consists of a single curve (modulo isomorphism), which we denote by $X_\Gamma$. 
For all $\Gamma\in T_{1,n}$, define $\mathcal{U}_\Gamma\subset \Mmb 1n$ as the open substack obtained by taking all curves $X$ that specialise to $X_\Gamma$, i.e. those
whose dual graph admits a morphism $\Gamma(X)\rightarrow \Gamma$.
Since the curves $X_\Gamma$ are the deepest strata of the stratification of $\Mmb 1n$ by topological type,
the set $\{\mathcal U_\Gamma\}_{\Gamma \in T_{1,n}}$ is an open cover of~$\Mmb 1n$. %

For each $\Gamma\in T_{1,n}$, by Proposition~\ref{singleispolarised}, there exists  a nondegenerate $\phi_\Gamma'\in V^d(\Gamma)$ such that the collection of  multidegrees of line bundles in $\overline{\mathcal J}^d_{\phi'_\Gamma}(X_\Gamma)$ equals $S_{\Gamma}$. By Proposition~\ref{surjectivegenus1}, we can choose a nondegenerate $\phi_\Gamma \in V_{1,n}^d$ that extends $\phi_\Gamma'$. Since $\{(\Gamma,S_\Gamma)\}$ is a compatible assignment, we have that for every curve $X$ in $\mathcal U_\Gamma$ the collection of multidegrees of line bundles in  $\overline{\mathcal J}^d_{\phi_\Gamma(\Gamma(X))}(X)$ equals $S_{\Gamma(X)}$. Hence we define $\mathcal{Y}|_{\mathcal{U}_\Gamma}$ as the restriction $\overline{\mathcal J}^d_{1,n}(\phi_\Gamma)|_{\mathcal{U}_\Gamma}$. This immediately implies that $\mathcal Y$ is open in $\Simp d {\Cmb 1n / \Mmb 1n}$.
Moreover, since $\mathcal{Y}|_{\mathcal{U}_\Gamma}=\overline{\mathcal J}^d_{1,n}(\phi_\Gamma)|_{\mathcal{U}_\Gamma} \to \mathcal{U}_\Gamma$ is proper,  by 
\cite[\href{https://stacks.math.columbia.edu/tag/01W0}{Lemma~29.41.3}]{stacks} we obtain that the representable morphism $\mathcal{Y}\to \Mmb 1n$ is proper as well.
\end{proof}

In order to prove Theorem~\ref{f-and-g}, we next tackle the problem of finding a combinatorial description of the restriction of each compatible assignment~\eqref{correspondence} %
to the rational tails of each topological type $\Gamma$. %
We will do this by associating with  each assignment 
the function $g\colon \mathcal Q^+_n\rightarrow \Z$ obtained by taking the degree on the genus $0$ component of every $2$-component curve with $1$ node, as described in Definition~\ref{def:fandg}.

\begin{lemma}\label{g-for-rattails} 
The map 
\[\Jmb d1n \longmapsto \left(\left\{(\Gamma,S_\Gamma(\Jmb d1n))\co\Gamma \in G_{1,n}^{\NRop} \right\}, g_{\Jmb d1n}\right)\]
defined using Definition~\ref{sgamma} and Definition~\ref{def:fandg} gives a bijection between the set of all degree~$d$ \fcujs{} $\Jmb d1n$ over $\Mmb 1n$ and the set of pairs   %
\[\left(\left\{(\Gamma,S_\Gamma)\co\Gamma \in G_{1,n}^{\NRop}, S_\Gamma\subset \mathbb{Z}^{\operatorname{Vert}(\Gamma)} \right\}, g\colon \mathcal Q_n^+\rightarrow \Z\right)\]
where the first component is a compatible assignment of line bundle multidegrees over all topological types of necklace curves and the second an arbitrary function.
\end{lemma}

\begin{proof}
We want to show that compatibility allows to uniquely reconstruct the collection of line bundle multidegrees of an arbitrary graph $\Gamma\in G_{1,n}$ from the line bundle multidegrees over the graph $\Gamma'$ obtained after contracting the rational tails of $\Gamma$, and the assignment $g$.

The argument to show that the values on the rational tails can be uniquely reconstructed from the function $g$ is the same as the argument in  \cite[Lemma~3.9]{kp2}. We want to show that there is a unique extension of each assignment $\bf d'$ of a line bundle multidegree over $\Gamma'$ to an assignment $\bf d$ of a line bundle multidegree over $\Gamma$ that is compatible with the assignment of line bundle bidegrees over all curves with $2$ nonsingular components and $1$ node (which is the information encoded in the function $g$).  

If $l$ is the number of vertices of $\Gamma$ contained in some rational tail, extending $\bf d'$ corresponds to determining the value of $l$ extra integers. A set of $l$ affine linear constraints is obtained by imposing that the total degree equals $d$ ($1$ constraint) and from compatibility by considering the $l-1$ morphisms each of which contracts all except $1$ of the edges contained in the rational tails ($l-1$ constraints). The proof of \cite[Lemma~3.9]{kp2} shows that this linear system has determinant $\pm1$. This implies that there exists a unique extension of each line bundle multidegree $\bf d'$ for $\Gamma'$ to a line bundle multidegree $\bf d$ for $\Gamma$ that is compatible with $g$, which proves our statement.  
\end{proof}

We are now ready to prove the main result. 

\begin{proof}[Proof of Theorem~\ref{f-and-g}]
After Lemma~\ref{g-for-rattails}, 
it only remains to study the combinatorial data needed to define the restriction of a degree~$d$ \fcuj to the moduli stack $\NR n$ of stable curves of genus $1$ without rational tails, i.e. necklace curves. Equivalently, it is enough to understand how to construct a compatible assignment  
$\left\{(\Gamma,S_\Gamma)\co\Gamma \in G_{1,n}^{\NRop}, S_\Gamma\subset \mathbb{Z}^{\operatorname{Vert}(\Gamma)} \right\}$
of line bundle multidegrees on necklace curves. %

Let $T_{1,n}^{\NRop} \subset G_{1,n}^{\NRop}$ be the subset of graphs with $n$ edges, i.e.
 stable necklace graphs with $n$ vertices and $n$ labelled half-edges. %
By compatibility, the restriction of a \fcuj{} $\Jmb d1n$ to curves without rational tails is uniquely determined by the assignment 
$\left\{(\Gamma,S_\Gamma)\co\Gamma \in T_{1,n}^{\NRop}, S_\Gamma\subset \mathbb{Z}^{\operatorname{Vert}(\Gamma)} \right\}$
for all necklace graphs $\Gamma\in T_{1,n}^{\NRop}$, subject to the compatibility condition that for every pair $\Gamma,\Gamma'\in T_{1,n}^{\NRop}$ of necklace graphs with $n$ components and every two morphisms 
$\Gamma\xrightarrow{\alpha} \Gamma'' \xleftarrow{\alpha'}\Gamma'$ with the same target, we have that for every $\mathbf d\in S_{\Gamma}$ there exists a $\mathbf d'\in S_{\Gamma'}$ such that 
\begin{equation}\label{compexpl}
\sum_{v\in\alpha^{-1}(w)} d(v) = \sum_{v\in{\alpha'}^{-1}(w)} d'(v) \ \text{ for all }w\in\Vrt(\Gamma'').
\end{equation}

Let us review the results of Section~\ref{genus1isolated} on how to describe a collection $S_\Gamma\subset \Z^{\Vrt(\Gamma)}$ of line bundle multidegrees of a degree~$d$ \fcj on a stable curve of topological type $\Gamma\in T_{1,n}^{\NRop}$. 
Since each graph $\Gamma \in T_{1,n}^{\NRop}$ has exactly $1$ half-edge on each vertex, we have an identification $\Z^{\Vrt(\Gamma)}\cong \Z^n$. 
Moreover, each $\Gamma$ corresponds to a cyclic ordering of the labels $1,\dots,n$ of the half-edges.

Consider a necklace curve $C$ in $\Mmb 1n$ of type $\Gamma\in T_{1,n}^{\NRop}$, corresponding to a cyclic ordering 
\[i_0{\prec} i_1{\prec} \cdots{\prec} i_{n-1}{\prec} i_n=i_0\] of the half-edges of $\Gamma$, i.e. the marked points that identify the irreducible components of $C$. (Choosing $ \succ$ rather than ${\prec}$ would give rise to an equivalent construction).  Proposition~\ref{f-for-fcjs} provides a bijection between the degree~$d$ \fcjs{} of $C$ and the functions 
\[
\begin{array}{rrcl}
f_{\overline{\mathcal J}^d(C)}\colon&\mathcal C_{n-1,\prec}&\longrightarrow&\Z
\end{array}
\]
defined on the set $\mathcal C_{n-1,\prec}:=\{I=\{i_{r},i_{r+1},\dots,i_s\}: 1\leq r\leq s\leq n-1\}$ of sequences of integers that are consecutive for the ordering ${\prec}$ that are mildly superadditive in the sense of \eqref{superadd}
for all disjoint sequences $I, J$ of  ${\prec}$-consecutive integers such that $I\cup J\in \mathcal C_{n-1,\prec}$. 

By the previous paragraph, the datum of an assignment  $\left\{(\Gamma,S_\Gamma)\colon\Gamma \in T_{1,n}^{\NRop}, S_\Gamma\subset \mathbb{Z}^{\operatorname{Vert}(\Gamma)}\right\}$  %
can be encoded in the datum of an assignment of functions $f_\Gamma$ for all $\Gamma\in T_{1,n}^{\NRop}$ or, equivalently, for all cyclic orderings $\prec$ of the set $\{1,\dots,n\}$.
Comparing the compatibility condition~\eqref{compexpl} with the formula for $f_\Gamma = f_{\overline{\mathcal J}^d(C)}$ given in Proposition~\ref{f-for-fcjs} yields that 
for every subset $I\subseteq\{1,\dots,n-1\}$, the value $f_{\Gamma}(I)$
 should be the same for all $\Gamma$ for which it is defined, i.e. for all cyclic orderings $\prec$ for which the elements of $I$ can be ordered in a consecutive way. This happens if and only if the dual graph $\Gamma$ admits a morphism to the dual graph of a curve $X_I$ as in Definition~\ref{def:fandg}.
As a consequence, we can glue together the functions $f_\Gamma\co \mathcal C_{n-1,\prec} \rightarrow \Z$
to a unique function $f$ defined on all nonempty subsets of $\{1,\dots,n-1\}$, which satisfies \eqref{superadd} since for every two disjoint subsets $I,J\in\mathcal P^+_{n-1}$ we can find a cyclic order ${\prec}$ such that $I,J,I\cup J\in \mathcal C_{n-1,\prec}$. 
The characterisation of $f$ given in Definition~\ref{def:fandg} is obtained by compatibility after contracting all edges of a graph $\Gamma \in T_{1,n}^{\NRop}$ joining either two vertices in $I$ or two vertices in $I^c$.
\end{proof}

Analogously to the situation for \fcjs{} of a single genus~$1$ curve, one may wonder if every \fcuj of genus $1$ can   be defined by some universal polarisation $\phi\in V^d_{1,n}$ (as defined in Section~\ref{notallpolar}). %
Below we  give examples to show that this is not always the case. In preparation for that, we recall the notation for universal polarisations and describe how a universal polarisation $\phi$ determines the mildly superadditive function $f$ associated with $\Jmb d1n(\phi)$.

\begin{remark}\label{rem:formulaforf}
The function $f$ arising from a \fcuj of the form $\Jmb d1n(\phi)$ for some nondegenerate $\phi\in V^d_{1,n}$ has a specific form in terms of the coordinates \[\phi=(x_1, \ldots, x_{n-1}, y_1, \ldots, y_{2^n-n-1})\] introduced in Proposition~\ref{polytopes} on the space of universal genus $1$ polarisations. For each $1 \leq i \leq n-1$ our choice of coordinate was fixed so that $\phi(\Gamma(X_i))=(x_i, d-x_i)$, where $X_i$ is any necklace curve with two components, and the first component is the one that carries the $i$-th marking alone.   

By the compatibility of $\phi$ with graph morphisms (Definition~\ref{phicompatible}), this generalises to any nonempty subset $I\subseteq\{1,\dots,n-1\}$. Namely, if we consider the necklace curve $X_I$ with two components $C_{I,1}$ and $C_{I,2}$, where $C_{I,1}$ is marked exactly with the points in $I$, then we have $\phi(\Gamma(X_I))_{C_{I,1}}=\sum_{i\in I}x_i$. By Definition~\ref{def:fandg}, the two degrees $f_{\Jmb d1n}(I)$ and $f_{\Jmb d1n}(I)+1$ are $\phi$-stable on the component $C_{I,1}$ of $X_I$, which by Definition~\ref{phistab} can be rephrased as 
\begin{equation}\label{mildsuperphi}
f_{\Jmb d1n(\phi)}(I)<\sum_{i\in I}x_i<f_{\Jmb d1n(\phi)}(I)+1.
\end{equation}
\end{remark}

We are now ready to produce examples of fine compactified universal Jacobians that are not arising from $\phi$-stability for any universal polarisation $\phi \in V_{1,n}^d$.

\begin{example}\label{exotic}
Fix integers $d\in \mathbb{Z}$ and
 $n\geq 6$ and set $g\co \mathcal Q^+_n\rightarrow\Z$ to be any sum zero function (for example, the zero function). Let $f\colon \mathcal P_{n-1}^+\rightarrow\Z$ be the function defined by
\begin{equation} \label{strangef}
f(I) = \left\{
\begin{array}{ll}
1&\text{if }\{1,3,5\}\subseteq I \text{ or }\{2,4,5\}\subseteq I\\
0&\text{else.}
\end{array}
\right.
\end{equation}

It is straightforward to check that the function $f$ is mildly superadditive as prescribed in \eqref{superadd}. We claim that for the \fcuj $\Jmb d1n$ that corresponds to the pair $(f,g)$ via Theorem~\ref{f-and-g}, there is no $\phi \in V_{1,n}^d$ such that $\Jmb d1n=\Jmb d1n(\phi)$.

Indeed let $\phi \in V_{1,n}^d$ and fix coordinates as in  Proposition~\ref{polytopes} so that $\phi=(x_1, \ldots, x_{n-1}, y_1, \ldots, y_{2^n-n-1})$. Then if we consider the constraints imposed by~\eqref{mildsuperphi} with  $I=\{1,3,5\}$ and  with $I=\{2,3,5\}$, we obtain  
\[
\left.
\begin{array}{l@{}l}
x_1+x_3+x_5 &> 1\\
x_2+x_3+x_5 &< 1\\
\end{array}
\right\} \Rightarrow x_1>x_2,
\]
whereas for  $I=\{1,4,5\}$ and $I=\{2,4,5\}$, we obtain 
\[
\left.
\begin{array}{l@{}l}
x_1+x_4+x_5 &< 1\\
x_2+x_4+x_5 &> 1\\
\end{array}
\right\} \Rightarrow x_1<x_2,
\]
which together yield a contradiction.

In Figure~\ref{9graphs} we show the degree $d=2$ polarisations on necklace curves with $n=6$ components that appear in $\Jmb 216$, up to the action of the subset of the symmetric group $\s_6$ generated by the permutations $(1\;3)$, $(2\;4)$ and $(1\;2)(3\;4)$.

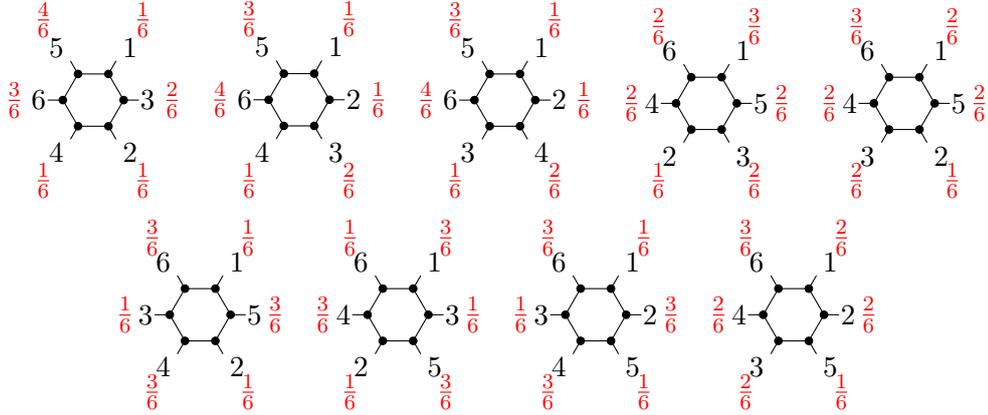
\begin{figure}[t!]
\begin{tikzpicture}
\small
   \newdimen\R
   \R=0.4cm
   \newdimen\T
   \T=0.35cm
   \newdimen\r
   \r=0.65cm
   \newdimen\S
   \S=0.6cm
   \draw (0:\R) \foreach \x in {60,120,...,360} {  -- (\x:\R) };
   \foreach \x/\l/\m in
     { 60/{1}/{$\frac16$},
      360/{3}/{$\frac26$},
      300/{2}/{$\frac16$},
      240/{4}/{$\frac16$},
      180/{6}/{$\frac36$},
      120/{5}/{$\frac46$}
     }
     {\node[inner sep=1pt,circle,draw,fill] at (\x:\R) {};
     \draw (\x:\R) -- (\x:\S);
     \node[label={\x:\l}] at (\x:\T) {};
     \node[label={[red]\x:\m}] at (\x:\r) {};}
\end{tikzpicture}
\begin{tikzpicture}
\small
   \newdimen\R
   \R=0.4cm
   \newdimen\T
   \T=0.35cm
   \newdimen\r
   \r=0.655cm
   \newdimen\S
   \S=0.6cm
   \draw (0:\R) \foreach \x in {60,120,...,360} {  -- (\x:\R) };
   \foreach \x/\l/\m in
     { 60/{1}/{$\frac16$},
      360/{2}/{$\frac16$},
      300/{3}/{$\frac26$},
      240/{4}/{$\frac16$},
      180/{6}/{$\frac46$},
      120/{5}/{$\frac36$}
     }
     {\node[inner sep=1pt,circle,draw,fill] at (\x:\R) {};
     \draw (\x:\R) -- (\x:\S);
     \node[label={\x:\l}] at (\x:\T) {};
     \node[label={[red]\x:\m}] at (\x:\r) {};}
\end{tikzpicture}
\begin{tikzpicture}
\small
   \newdimen\R
   \R=0.4cm
   \newdimen\T
   \T=0.35cm
   \newdimen\r
   \r=0.65cm
   \newdimen\S
   \S=0.6cm
   \draw (0:\R) \foreach \x in {60,120,...,360} {  -- (\x:\R) };
   \foreach \x/\l/\m in
     { 60/{1}/{$\frac16$},
      360/{2}/{$\frac16$},
      300/{4}/{$\frac26$},
      240/{3}/{$\frac16$},
      180/{6}/{$\frac46$},
      120/{5}/{$\frac36$}
     }
     {\node[inner sep=1pt,circle,draw,fill] at (\x:\R) {};
     \draw (\x:\R) -- (\x:\S);
     \node[label={\x:\l}] at (\x:\T) {};
     \node[label={[red]\x:\m}] at (\x:\r) {};}
\end{tikzpicture}
\begin{tikzpicture}
\small
   \newdimen\R
   \R=0.4cm
   \newdimen\T
   \T=0.35cm
   \newdimen\r
   \r=0.6cm
   \newdimen\S
   \S=0.6cm
   \draw (0:\R) \foreach \x in {60,120,...,360} {  -- (\x:\R) };
   \foreach \x/\l/\m in
     { 60/{1}/{$\frac36$},
      360/{5}/{$\frac26$},
      300/{3}/{$\frac26$},
      240/{2}/{$\frac16$},
      180/{4}/{$\frac26$},
      120/{6}/{$\frac26$}
     }
     {\node[inner sep=1pt,circle,draw,fill] at (\x:\R) {};
     \draw (\x:\R) -- (\x:\S);
     \node[label={\x:\l}] at (\x:\T) {};
     \node[label={[red]\x:\m}] at (\x:\r) {};}
\end{tikzpicture}
\begin{tikzpicture}
\small
   \newdimen\R
   \R=0.4cm
   \newdimen\T
   \T=0.35cm
   \newdimen\r
   \r=0.6cm
   \newdimen\S
   \S=0.6cm
   \draw (0:\R) \foreach \x in {60,120,...,360} {  -- (\x:\R) };
   \foreach \x/\l/\m in
     { 60/{1}/{$\frac26$},
      360/{5}/{$\frac26$},
      300/{2}/{$\frac16$},
      240/{3}/{$\frac26$},
      180/{4}/{$\frac26$},
      120/{6}/{$\frac36$}
     }
     {\node[inner sep=1pt,circle,draw,fill] at (\x:\R) {};
     \draw (\x:\R) -- (\x:\S);
     \node[label={\x:\l}] at (\x:\T) {};
     \node[label={[red]\x:\m}] at (\x:\r) {};}
\end{tikzpicture}
\begin{tikzpicture}
\small
   \newdimen\R
   \R=0.4cm
   \newdimen\T
   \T=0.35cm
   \newdimen\r
   \r=0.6cm
   \newdimen\S
   \S=0.6cm
   \draw (0:\R) \foreach \x in {60,120,...,360} {  -- (\x:\R) };
   \foreach \x/\l/\m in
     { 60/{1}/{$\frac16$},
      360/{5}/{$\frac36$},
      300/{2}/{$\frac16$},
      240/{4}/{$\frac36$},
      180/{3}/{$\frac16$},
      120/{6}/{$\frac36$}
     }
     {\node[inner sep=1pt,circle,draw,fill] at (\x:\R) {};
     \draw (\x:\R) -- (\x:\S);
     \node[label={\x:\l}] at (\x:\T) {};
     \node[label={[red]\x:\m}] at (\x:\r) {};}
\end{tikzpicture}
\bigskip
\begin{tikzpicture}
\small
   \newdimen\R
   \R=0.4cm
   \newdimen\T
   \T=0.35cm
   \newdimen\r
   \r=0.6cm
   \newdimen\S
   \S=0.6cm
   \draw (0:\R) \foreach \x in {60,120,...,360} {  -- (\x:\R) };
   \foreach \x/\l/\m in
     { 60/{1}/{$\frac36$},
      360/{3}/{$\frac16$},
      300/{5}/{$\frac36$},
      240/{2}/{$\frac16$},
      180/{4}/{$\frac36$},
      120/{6}/{$\frac16$}
     }
     {\node[inner sep=1pt,circle,draw,fill] at (\x:\R) {};
     \draw (\x:\R) -- (\x:\S);
     \node[label={\x:\l}] at (\x:\T) {};
     \node[label={[red]\x:\m}] at (\x:\r) {};}
\end{tikzpicture}
\begin{tikzpicture}
\small
   \newdimen\R
   \R=0.4cm
   \newdimen\T
   \T=0.35cm
   \newdimen\r
   \r=0.6cm
   \newdimen\S
   \S=0.6cm
   \draw (0:\R) \foreach \x in {60,120,...,360} {  -- (\x:\R) };
   \foreach \x/\l/\m in
     { 60/{1}/{$\frac16$},
      360/{2}/{$\frac36$},
      300/{5}/{$\frac16$},
      240/{4}/{$\frac36$},
      180/{3}/{$\frac16$},
      120/{6}/{$\frac36$}
     }
     {\node[inner sep=1pt,circle,draw,fill] at (\x:\R) {};
     \draw (\x:\R) -- (\x:\S);
     \node[label={\x:\l}] at (\x:\T) {};
     \node[label={[red]\x:\m}] at (\x:\r) {};}
\end{tikzpicture}
\begin{tikzpicture}
\small
   \newdimen\R
   \R=0.4cm
   \newdimen\T
   \T=0.35cm
   \newdimen\r
   \r=0.6cm
   \newdimen\S
   \S=0.6cm
   \draw (0:\R) \foreach \x in {60,120,...,360} {  -- (\x:\R) };
   \foreach \x/\l/\m in
     { 60/{1}/{$\frac26$},
      360/{2}/{$\frac26$},
      300/{5}/{$\frac16$},
      240/{3}/{$\frac26$},
      180/{4}/{$\frac26$},
      120/{6}/{$\frac36$}
     }
     {\node[inner sep=1pt,circle,draw,fill] at (\x:\R) {};
     \draw (\x:\R) -- (\x:\S);
     \node[label={\x:\l}] at (\x:\T) {};
     \node[label={[red]\x:\m}] at (\x:\r) {};}
\end{tikzpicture}
\caption{\label{9graphs}Assignments of stability conditions on necklace graphs in Example~\ref{exotic} for $d=2$, $n=6$}
\end{figure}
\end{example}
We conclude with an observation on the problem of classification of fine compactified universal Jacobians modulo isomorphisms in light of the constructions of \cite[Section~6.2]{kp3}.

\begin{remark} \label{groupaction}
The collection of mildly superadditive functions $f\colon \mathcal{P}^+_{n} \to \mathbb{Z}$ can be interpreted, via Theorem~\ref{f-and-g}, as the equivalence classes of degree~$d$ genus $1$ fine compactified universal Jacobians modulo isomorphisms that extend the identity map $\mathcal{J}_{1,n}^d\to\mathcal{J}_{1,n}^d$ and that commute with the forgetful morphisms. This follows from the same argument given in \cite[Corollary~6.14]{kp3}.

Another natural question is to ask for the number of equivalence classes  of degree~$d$ genus $1$ fine compactified universal Jacobians modulo translation by some line bundle on $\Cmb 1n / \Mmb 1n$ of relative degree $0$. Theorem~\ref{f-and-g} gives a bijection between these equivalence classes and the collection  of mildly superadditive functions $f\colon \mathcal{P}^+_{n} \to \mathbb{Z}$ with the additional property that $f(\{i\})=0$ for all $1\leq i \leq n-1$.
Computing the number of these functions appears to be a challenging combinatorial problem.

\end{remark}

\bibliographystyle{alpha}
\bibliography{biblio-curves}

\begin{thebibliography}{BHPVdV04}

\bibitem[ACG11]{acg2}
Enrico Arbarello, Maurizio Cornalba, and Phillip~A. Griffiths.
\newblock {\em Geometry of algebraic curves. {V}olume {II}}, volume 268 of {\em
  Grundlehren der Mathematischen Wissenschaften [Fundamental Principles of
  Mathematical Sciences]}.
\newblock Springer, Heidelberg, 2011.
\newblock With a contribution by Joseph Daniel Harris.

\bibitem[AK80]{altman80}
Allen~B. Altman and Steven~L. Kleiman.
\newblock Compactifying the {P}icard scheme.
\newblock {\em Adv. in Math.}, 35(1):50--112, 1980.

\bibitem[Ale04]{alexeev-cjtm}
Valery Alexeev.
\newblock Compactified {J}acobians and {T}orelli map.
\newblock {\em Publ. Res. Inst. Math. Sci.}, 40(4):1241--1265, 2004.

\bibitem[BHPVdV04]{BHPvdV}
Wolf~P. Barth, Klaus Hulek, Chris A.~M. Peters, and Antonius Van~de Ven.
\newblock {\em Compact complex surfaces}, volume~4 of {\em Ergebnisse der
  Mathematik und ihrer Grenzgebiete. 3. Folge. A Series of Modern Surveys in
  Mathematics [Results in Mathematics and Related Areas. 3rd Series. A Series
  of Modern Surveys in Mathematics]}.
\newblock Springer-Verlag, Berlin, second edition, 2004.

\bibitem[BLR90]{blr}
Siegfried Bosch, Werner L\"{u}tkebohmert, and Michel Raynaud.
\newblock {\em N\'{e}ron models}, volume~21 of {\em Ergebnisse der Mathematik
  und ihrer Grenzgebiete (3) [Results in Mathematics and Related Areas (3)]}.
\newblock Springer-Verlag, Berlin, 1990.

\bibitem[Cap94]{caporaso}
Lucia Caporaso.
\newblock A compactification of the universal {P}icard variety over the moduli
  space of stable curves.
\newblock {\em J. Amer. Math. Soc.}, 7(3):589--660, 1994.

\bibitem[CD89]{cossec-dolgachev}
Fran\c{c}ois~R. Cossec and Igor~V. Dolgachev.
\newblock {\em Enriques surfaces. {I}}, volume~76 of {\em Progress in
  Mathematics}.
\newblock Birkh\"{a}user Boston, Inc., Boston, MA, 1989.

\bibitem[CMKV15]{cmkv1}
Sebastian Casalaina-Martin, Jesse~Leo Kass, and Filippo Viviani.
\newblock The local structure of compactified {J}acobians.
\newblock {\em Proc. Lond. Math. Soc. (3)}, 110(2):510--542, 2015.

\bibitem[DM69]{delignemumford}
P.~Deligne and D.~Mumford.
\newblock The irreducibility of the space of curves of given genus.
\newblock {\em Inst. Hautes \'{E}tudes Sci. Publ. Math.}, (36):75--109, 1969.

\bibitem[Est01]{esteves}
Eduardo Esteves.
\newblock Compactifying the relative {J}acobian over families of reduced
  curves.
\newblock {\em Trans. Amer. Math. Soc.}, 353(8):3045--3095, 2001.

\bibitem[Get98]{getzler-semiclassical}
E.~Getzler.
\newblock The semi-classical approximation for modular operads.
\newblock {\em Comm. Math. Phys.}, 194(2):481--492, 1998.

\bibitem[KP17]{kp2}
Jesse~Leo Kass and Nicola Pagani.
\newblock Extensions of the universal theta divisor.
\newblock {\em Adv. Math.}, 321:221--268, 2017.

\bibitem[KP19]{kp3}
Jesse~Leo Kass and Nicola Pagani.
\newblock The stability space of compactified universal {J}acobians.
\newblock {\em Trans. Amer. Math. Soc.}, 372(7):4851--4887, 2019.

\bibitem[KP22]{kasspa2}
Jesse~Leo Kass and Nicola Pagani.
\newblock Classifying fine universal {J}acobian stabilities.
\newblock 2022.
\newblock To appear.

\bibitem[Mel19]{melouniversal}
Margarida Melo.
\newblock Universal compactified {J}acobians.
\newblock {\em Port. Math.}, 76(2):101--122, 2019.

\bibitem[MRV17]{meravi}
Margarida Melo, Antonio Rapagnetta, and Filippo Viviani.
\newblock Fine compactified {J}acobians of reduced curves.
\newblock {\em Trans. Amer. Math. Soc.}, 369(8):5341--5402, 2017.

\bibitem[MSV21]{mishvi}
Luca Migliorini, Vivek Shende, and Filippo Viviani.
\newblock A support theorem for {H}ilbert schemes of planar curves, {II}.
\newblock {\em Compos. Math.}, 157(4):835--882, 2021.

\bibitem[OS79]{oda79}
Tadao Oda and C.~S. Seshadri.
\newblock Compactifications of the generalized {J}acobian variety.
\newblock {\em Trans. Amer. Math. Soc.}, 253:1--90, 1979.

\bibitem[Pan96]{panda}
Rahul Pandharipande.
\newblock A compactification over {$\overline {M}_g$} of the universal moduli
  space of slope-semistable vector bundles.
\newblock {\em J. Amer. Math. Soc.}, 9(2):425--471, 1996.

\bibitem[Pet12]{dan-semiclassical}
Dan Petersen.
\newblock A remark on {G}etzler's semi-classical approximation.
\newblock In {\em Geometry and arithmetic}, EMS Ser. Congr. Rep., pages
  309--316. Eur. Math. Soc., Z\"{u}rich, 2012.

\bibitem[Pet14]{petersen}
Dan Petersen.
\newblock The structure of the tautological ring in genus one.
\newblock {\em Duke Math. J.}, 163(4):777--793, 2014.

\bibitem[{Sag}17]{sagemath}
The {Sage Developers}.
\newblock {\em {S}ageMath, the {S}age {M}athematics {S}oftware {S}ystem
  ({V}ersion 8.1)}, 2017.

\bibitem[Sim94]{simpson}
C.~T. Simpson.
\newblock Moduli of representations of the fundamental group of a smooth
  projective variety. {I}.
\newblock {\em Inst. Hautes \'Etudes Sci. Publ. Math.}, 79:47--129, 1994.

\bibitem[{Sta}22]{stacks}
The {Stacks Project Authors}.
\newblock \textit{Stacks Project}.
\newblock \url{https://stacks.math.columbia.edu}, 2022.

\end{thebibliography}

\begin{table}[b]
{\tiny
     \begin{tabular}{ll}
     $
    \begin{array}[t]{lll}
    n  & \lambda & \langle \hodgechar n {\Jmb d1n},s_\lambda\rangle\\\hline
1&  {1}      &L^2 + 2L + 1\\
2&  {2}      &L^3 + 3L^2 + 3L + 1\\
&  {1, 1}   &L^2 + L\\
3&  {3}      &L^4 + 4L^3 + 7L^2 + 4L + 1\\
&  {2, 1}   &2L^3 + 4L^2 + 2L\\
4&  {4}      &L^5 + 5L^4 + 13L^3 + 13L^2 + 5L + 1\\
&  {3, 1}   &3L^4 + 10L^3 + 10L^2 + 3L\\[1ex]
  n & \lambda & \langle\hodgechar n {\Mmb 1n},s_{\lambda}\rangle\\\hline
2&  {2}&       L^2 + 2L + 1\\
3&  {3}&       L^3 + 3L^2 + 3L + 1\\
&  {2, 1}&    L^2 + L\\
4&  {4}&       L^4 + 4L^3 + 7L^2 + 4L + 1\\
&  {3, 1}&    2L^3 + 4L^2 + 2L\\
&  {2, 2}&    L^3 + 2L^2 + L\\
5&  {5}&       L^5 + 5L^4 + 12L^3 + 12L^2 + 5L + 1\\
&  {4, 1}&    3L^4 + 11L^3 + 11L^2 + 3L\\
&  {3, 2}&    2L^4 + 7L^3 + 7L^2 + 2L\\
\end{array}$
&
     $\begin{array}[t]{lll}
    n  & \lambda & \langle \hodgechar n {\Jmb d1n},s_\lambda\rangle\\\hline
4&  {2, 2}   &L^4 + 4L^3 + 4L^2 + L\\
&  {2, 1, 1}&L^3 + L^2\\
5&  {5}      &L^6 + 6L^5 + 21L^4 + 31L^3 + 21L^2 + 6L + 1\\
&  {4, 1}   &4L^5 + 20L^4 + 34L^3 + 20L^2 + 4L\\
&  {2, 2, 1}&2L^4 + 4L^3 + 2L^2\\
&  {3, 2}   &2L^5 + 12L^4 + 21L^3 + 12L^2 + 2L\\
&  {3, 1, 1}&3L^4 + 7L^3 + 3L^2
 \\[1ex]
 n & \lambda & \langle\hodgechar n {\Mmb 1n},s_{\lambda}\rangle\\\hline
5&  {3, 1, 1}& L^3 + L^2\\
&  {2, 2, 1}& L^3 + L^2\\
6&  {6}&       L^6 + 6L^5 + 20L^4 + 29L^3 + 20L^2 + 6L + 1\\
&  {5, 1}&    4L^5 + 20L^4 + 33L^3 + 20L^2 + 4L\\
&  {4, 2}&    3L^5 + 18L^4 + 31L^3 + 18L^2 + 3L\\
&  {4, 1, 1}& 4L^4 + 9L^3 + 4L^2\\
&  {3, 3}&    L^5 + 6L^4 + 11L^3 + 6L^2 + L\\
&  {3, 2, 1}& 4L^4 + 8L^3 + 4L^2\\
&  {2, 2, 2}& L^4 + 2L^3 + L^2\\
\end{array}$
\end{tabular}}
\caption{Hodge Euler characteristic of $\Jmb d1n$ and $\Mmb 1{n+1}$ for $n\leq 5$.   \label{table-jmb-1}}
\end{table}

\begin{table}
{\tiny
$\begin{array}{lll}
    n  & \lambda & \langle \hodgechar n {\Jmb d1n},s_\lambda\rangle\\\hline
6&  {6}            &L^7 + 7L^6 + 31L^5 + 63L^4 + 63L^3 + 31L^2 + 7L + 1\\
&  {5, 1}         &5L^6 + 33L^5 + 82L^4 + 82L^3 + 33L^2 + 5L\\
&  {4, 2}         &3L^6 + 26L^5 + 71L^4 + 71L^3 + 26L^2 + 3L\\
&  {4, 1, 1}      &7L^5 + 27L^4 + 27L^3 + 7L^2\\
&  {3, 3}         &L^6 + 9L^5 + 26L^4 + 26L^3 + 9L^2 + L\\
&  {3, 2, 1}      &6L^5 + 23L^4 + 23L^3 + 6L^2\\
&  {3, 1, 1, 1}   &L^4 + L^3\\
&  {2, 2, 2}      &L^5 + 4L^4 + 4L^3 + L^2\\
&  {2, 2, 1, 1}   &L^4 + L^3\\
7&  {7}            &L^8 + 8L^7 + 42L^6 + 110L^5 + 154L^4 + 110L^3 + 42L^2 + 8L + 1\\
&  {6, 1}         &6L^7 + 50L^6 + 167L^5 + 247L^4 + 167L^3 + 50L^2 + 6L\\
&  {5, 2}         &4L^7 + 44L^6 + 168L^5 + 262L^4 + 168L^3 + 44L^2 + 4L\\
&  {5, 1, 1}      &12L^6 + 67L^5 + 112L^4 + 67L^3 + 12L^2\\
&  {4, 3}         &2L^7 + 26L^6 + 108L^5 + 172L^4 + 108L^3 + 26L^2 + 2L\\
&  {4, 2, 1}      &14L^6 + 83L^5 + 146L^4 + 83L^3 + 14L^2\\
&  {4, 1, 1, 1}   &6L^5 + 14L^4 + 6L^3\\
&  {3, 3, 1}      &5L^6 + 33L^5 + 58L^4 + 33L^3 + 5L^2\\
&  {3, 2, 2}      &3L^6 + 20L^5 + 37L^4 + 20L^3 + 3L^2\\
&  {3, 2, 1, 1}   &6L^5 + 14L^4 + 6L^3\\
&  {2, 2, 2, 1}   &2L^5 + 4L^4 + 2L^3\\
8&  {8}            &L^9 + 9L^8 + 55L^7 + 177L^6 + 322L^5 + 322L^4 + 177L^3 + 55L^2 + 9L + 1\\
&  {7, 1}         &7L^8 + 69L^7 + 293L^6 + 596L^5 + 596L^4 + 293L^3 + 69L^2 + 7L\\
&  {6, 2}         &5L^8 + 68L^7 + 337L^6 + 739L^5 + 739L^4 + 337L^3 + 68L^2 + 5L\\
&  {6, 1, 1}      &19L^7 + 139L^6 + 342L^5 + 342L^4 + 139L^3 + 19L^2\\
&  {5, 3}         &3L^8 + 48L^7 + 267L^6 + 618L^5 + 618L^4 + 267L^3 + 48L^2 + 3L\\
&  {5, 2, 1}      &24L^7 + 199L^6 + 530L^5 + 530L^4 + 199L^3 + 24L^2\\
&  {5, 1, 1, 1}   &16L^6 + 62L^5 + 62L^4 + 16L^3\\
&  {4, 4}         &L^8 + 19L^7 + 109L^6 + 257L^5 + 257L^4 + 109L^3 + 19L^2 + L\\
&  {4, 3, 1}      &16L^7 + 145L^6 + 401L^5 + 401L^4 + 145L^3 + 16L^2\\
&  {4, 2, 2}      &7L^7 + 69L^6 + 202L^5 + 202L^4 + 69L^3 + 7L^2\\
&  {4, 2, 1, 1}   &25L^6 + 100L^5 + 100L^4 + 25L^3\\
&  {4, 1, 1, 1, 1}&2L^5 + 2L^4\\
&  {3, 3, 2}      &3L^7 + 34L^6 + 103L^5 + 103L^4 + 34L^3 + 3L^2\\
&  {3, 3, 1, 1}   &10L^6 + 42L^5 + 42L^4 + 10L^3\\
&  {3, 2, 2, 1}   &9L^6 + 37L^5 + 37L^4 + 9L^3\\
&  {3, 2, 1, 1, 1}&2L^5 + 2L^4\\
&  {2, 2, 2, 2}   &L^6 + 4L^5 + 4L^4 + L^3\\
&  {2, 2, 2, 1, 1}&L^5 + L^4\\
  \end{array}$
  }
\caption{Hodge Euler characteristic of $\Jmb d1n$ for $6\leq n\leq 8$.   \label{table-jmb}}
\end{table}

\begin{table}
  {\tiny
     $\begin{array}{lll}
 n & \lambda & \langle\hodgechar n {\Mmb 1n},s_{\lambda}\rangle\\\hline
7&  {7}&             L^7 + 7L^6 + 28L^5 + 56L^4 + 56L^3 + 28L^2 + 7L + 1\\
&  {6, 1}&          5L^6 + 34L^5 + 81L^4 + 81L^3 + 34L^2 + 5L\\
&  {5, 2}&          4L^6 + 32L^5 + 85L^4 + 85L^3 + 32L^2 + 4L\\
&  {5, 1, 1}&       8L^5 + 29L^4 + 29L^3 + 8L^2\\
&  {4, 3}&          2L^6 + 20L^5 + 56L^4 + 56L^3 + 20L^2 + 2L\\
&  {4, 2, 1}&       11L^5 + 41L^4 + 41L^3 + 11L^2\\
&  {4, 1, 1, 1}&    2L^4 + 2L^3\\
&  {3, 3, 1}&       4L^5 + 16L^4 + 16L^3 + 4L^2\\
&  {3, 2, 2}&       3L^5 + 11L^4 + 11L^3 + 3L^2\\
&  {3, 2, 1, 1}&    2L^4 + 2L^3\\
&  {2, 2, 2, 1}&    L^4 + L^3\\
8&  {8}&             L^8 + 8L^7 + 39L^6 + 98L^5 + 136L^4 + 98L^3 + 39L^2 + 8L + 1\\
&  {7, 1}&          6L^7 + 49L^6 + 159L^5 + 232L^4 + 159L^3 + 49L^2 + 6L\\
&  {6, 2}&          5L^7 + 53L^6 + 193L^5 + 294L^4 + 193L^3 + 53L^2 + 5L\\
&  {6, 1, 1}&       14L^6 + 73L^5 + 119L^4 + 73L^3 + 14L^2\\
&  {5, 3}&          3L^7 + 38L^6 + 155L^5 + 243L^4 + 155L^3 + 38L^2 + 3L\\
&  {5, 2, 1}&       20L^6 + 115L^5 + 195L^4 + 115L^3 + 20L^2\\
&  {5, 1, 1, 1}&    8L^5 + 17L^4 + 8L^3\\
&  {4, 4}&          L^7 + 16L^6 + 66L^5 + 105L^4 + 66L^3 + 16L^2 + L\\
&  {4, 3, 1}&       14L^6 + 87L^5 + 150L^4 + 87L^3 + 14L^2\\
&  {4, 2, 2}&       7L^6 + 45L^5 + 80L^4 + 45L^3 + 7L^2\\
&  {4, 2, 1, 1}&    14L^5 + 30L^4 + 14L^3\\
&  {3, 3, 2}&       3L^6 + 22L^5 + 39L^4 + 22L^3 + 3L^2\\
&  {3, 3, 1, 1}&    6L^5 + 13L^4 + 6L^3\\
&  {3, 2, 2, 1}&    6L^5 + 12L^4 + 6L^3\\
&  {2, 2, 2, 2}&    L^5 + 2L^4 + L^3\\
9 &  {9}&             L^9 + 9L^8 + 50L^7 + 157L^6 + 278L^5 + 278L^4 + 157L^3 + 50L^2 + 9L + 1\\
&  {8, 1}&          7L^8 + 69L^7 + 279L^6 + 554L^5 + 554L^4 + 279L^3 + 69L^2 + 7L\\
&  {7, 2}&          6L^8 + 76L^7 + 364L^6 + 775L^5 + 775L^4 + 364L^3 + 76L^2 + 6L\\
&  {7, 1, 1}&       21L^7 + 147L^6 + 351L^5 + 351L^4 + 147L^3 + 21L^2\\
&  {6, 3}&          4L^8 + 65L^7 + 348L^6 + 779L^5 + 779L^4 + 348L^3 + 65L^2 + 4L\\
&  {6, 2, 1}&       33L^7 + 259L^6 + 657L^5 + 657L^4 + 259L^3 + 33L^2\\
&  {6, 1, 1, 1}&    21L^6 + 74L^5 + 74L^4 + 21L^3\\
&  {5, 4}&          2L^8 + 37L^7 + 214L^6 + 497L^5 + 497L^4 + 214L^3 + 37L^2 + 2L\\
&  {5, 3, 1}&       27L^7 + 242L^6 + 647L^5 + 647L^4 + 242L^3 + 27L^2\\
&  {5, 2, 2}&       12L^7 + 116L^6 + 320L^5 + 320L^4 + 116L^3 + 12L^2\\
&  {5, 2, 1, 1}&    42L^6 + 155L^5 + 155L^4 + 42L^3\\
&  {5, 1, 1, 1, 1}& 3L^5 + 3L^4\\
&  {4, 4, 1}&       12L^7 + 109L^6 + 293L^5 + 293L^4 + 109L^3 + 12L^2\\
&  {4, 3, 2}&       10L^7 + 109L^6 + 315L^5 + 315L^4 + 109L^3 + 10L^2\\
&  {4, 3, 1, 1}&    35L^6 + 132L^5 + 132L^4 + 35L^3\\
&  {4, 2, 2, 1}&    25L^6 + 96L^5 + 96L^4 + 25L^3\\
&  {4, 2, 1, 1, 1}& 6L^5 + 6L^4\\
&  {3, 3, 3}&       L^7 + 15L^6 + 45L^5 + 45L^4 + 15L^3 + L^2\\
&  {3, 3, 2, 1}&    13L^6 + 51L^5 + 51L^4 + 13L^3\\
&  {3, 3, 1, 1, 1}& 3L^5 + 3L^4\\
&  {3, 2, 2, 2}&    4L^6 + 15L^5 + 15L^4 + 4L^3\\
&  {3, 2, 2, 1, 1}& 3L^5 + 3L^4\\
&  {2, 2, 2, 2, 1}& L^5 + L^4\\
\end{array}$
  }
  \caption{Hodge Euler characteristics of $\Mmb 1n$ for $7\leq n\leq 9$. \label{table-mmb}}
\end{table}

\end{document}